\documentclass[preprint,12pt,authoryear]{elsarticle}
\usepackage[margin=2.5cm]{geometry}
\usepackage{amssymb}
\usepackage[ruled,vlined]{algorithm2e} 
\usepackage{makecell}
\usepackage[table]{xcolor}
\usepackage{soul}
\usepackage{colortbl}
\usepackage{threeparttable}
\usepackage{verbatim}
\usepackage{stmaryrd}
\usepackage{amsmath}
\usepackage{natbib}
\usepackage{ragged2e}
\usepackage{multirow} 
\usepackage{array}
\usepackage{longtable,amsthm}
\usepackage{geometry}
\usepackage{adjustbox} 
\usepackage{booktabs} 
\usepackage{pdflscape}
\usepackage[normalem]{ulem}
\usepackage{tikz}
\usepackage{pgfplots}
\pgfplotsset{compat=1.18}
\usepackage{subcaption}
\usepackage{graphicx}
\usepackage{url}
\usepackage[hyperfootnotes=false]{hyperref}
\usepackage{ulem}
\usepackage{setspace}
\newcolumntype{P}[1]{>{\centering\arraybackslash}m{#1}}
\usepackage{algpseudocode}

\newtheorem{mydef}{Definition}

\journal{Computers \& Operations Research}

\begin{document}
\onehalfspacing
\begin{frontmatter}


\title{Conflict-Aware Seat Assignment in Classroom Environments}


\author[inst2]{Bruna Cristina Braga Charytitsch}
\author[inst2]{Mari\' a Cristina Vasconcelos Nascimento}

\affiliation[inst2]{organization={Instituto Tecnológico de Aeronáutica (ITA)},
            addressline={Praça Marechal Eduardo Gomes, 50 - Vila das Acacias}, 
            city={São José dos Campos},
            postcode={12228-900}, 
            state={SP},
            country={Brazil}}

\begin{abstract}
\onehalfspacing

Determining an effective seating plan is a key challenge in classroom management, as the spatial arrangement of students directly influences interpersonal dynamics and the overall learning environment. This paper introduces the Student Seat Allocation Problem (SSAP), which consists of strategically assigning students to seats in a traditional classroom so as to minimize interpersonal conflicts while satisfying operational constraints. We propose a mathematical model and an Iterated Local Search (ILS) heuristic to address the SSAP. Computational experiments on a benchmark of 131 instances show that both methods converge to the optimal solution in 78 small- and medium-scale instances. For large-scale instances, Gurobi frequently exhausted the time limit yet achieved very low optimality gaps, outperforming the heuristic's average gap in 45 instances. ILS, in turn, outperformed Gurobi in 8 instances and, when analyzing individual runs, successfully reached a gap of zero in multiple executions for several large-scale instances — within a fraction of Gurobi's computational time. These results position ILS as an effective and accessible alternative for educational institutions where commercial solver licenses are unavailable or cost-prohibitive.

\end{abstract}


\end{frontmatter}



\section{Introduction }
\label{sec:sample1}
\onehalfspacing
Classroom dynamics are influenced by several factors that shape both teaching and the social environment. One of the first challenges for teachers is selecting the most effective seating plan to achieve lesson objectives. The layout — whether in rows, groups, pairs, modular configurations, or horseshoe shapes \citep{McCorskey1978, Kaya2007, Hamilton2019} — forms the foundation of the learning environment. The traditional row arrangement remains widely used and accepted in lecture-based lessons, where the teacher is the primary source of information. However, contemporary educational approaches view the teacher as a facilitator, encouraging students to take a more active role in their learning, often calling for more flexible seating arrangements.

In addition to teachers’ preferences, seating arrangements also impact students' comfort and motivation. Students often make individual choices based on their characteristics. For example, more extroverted students may feel at ease in group arrangements that promote interaction, while quieter students might prefer rows that offer less direct contact.

\cite{McCorskey1978} conducted a pioneering study investigating the influence of seating arrangements on classroom communication. This study focused on students' preferences for different layouts (traditional, horseshoe, and modular) and how these preferences related to their levels of communication apprehension (CA). Even decades later, McCorskey's findings remain relevant, suggesting that students with high CA prefer arrangements that inhibit interaction, while those with low CA tend to favor setups that facilitate participation. This seminal study underscores the importance of considering individual student preferences and the attractiveness of the course when deciding on seating arrangements. Allowing students the freedom to choose their seats can enhance communication and learning in the classroom.

Once the classroom seating arrangement is defined, another significant challenge for teachers is assigning students to specific seats. The literature dedicated numerous studies highlighting the challenges involved in this issue \citep{Meeks2013, Tobia2022, Yang2022, Nehyba2023}.
Various strategies take into account emotional, pedagogical, and social factors \citep{Gremmen2016, HOEKSTRA2023104016}. For example, seating choices might be based on student's physical and academic needs: taller students often sit in the back, while the positions of shorter students or those with learning difficulties must be closer to the front to facilitate interaction and focus. Teachers may also group students by the ability to encourage collaboration or separate them to prevent conflicts, fostering an environment that supports both learning and emotional well-being.

When assigning each student’s seating position, it is essential to consider conflict prevention, as conflicts can arise in various ways: from simple disagreements between peers to external issues, such as family rivalries. These conflicts may include cases of bullying or close friendships that disrupt the class due to excessive interactions. Careful planning of seating arrangements is crucial for maintaining harmony and focus in the classroom.

\cite{BRAUN2020} investigated the effects of a seating intervention on social dynamics among students, focusing on ``target" students (those displaying problematic behaviors) and ``nontarget" students (those without such behaviors). Conducted over eight weeks, the study involved placing students in proximity to those they had previously expressed dislike for. The findings revealed that, unlike previous interventions that showed positive short-term outcomes, this approach led to increased aggressive behaviors and reduced perceptions of cooperation among students, affecting both target and nontarget groups. The authors suggest that forced proximity may initially intensify conflicts and that a longer intervention period might be needed for relationships to stabilize and improve. This study underscores the complexity of classroom social dynamics and the importance of carefully analyzing behavior management interventions.

In Brazil’s basic education, teachers are often responsible for creating a ``seating chart", aiming to assign each student to a specific seat and periodically rearranging them to encourage new interactions and avoid a fixed arrangement. This process requires consensus among the teachers of the same class and reflects the complexity of managing classroom space effectively, especially within a traditional classroom layout. Despite ongoing debates about the most effective classroom setup, the traditional row arrangement remains common, particularly for maintaining control and facilitating instruction. Since student seats are periodically revised (e.g., every bimester), this allows teachers to update their decisions based on accumulated experience. Furthermore, regular staff meetings provide opportunities to share additional information about students, including behavioral issues, learning difficulties, and social dynamics that may not be immediately apparent to an individual teacher.

This paper introduces the Student Seat Allocation Problem (SSAP) for traditional classrooms, where seats are organized in parallel rows with a variable number of desks per row. The primary goal of SSAP is to assign students to seats in a way that minimizes pairwise student conflicts by placing students with potential issues as far apart as possible. 
This study focuses on upper elementary classrooms, where teachers often adjust the seating to manage conflicts and disruptions stemming from confrontational relationships or close friendships, making it essential to maximize the distance between students with conflicts while working within the physical constraints of the classroom.

This paper addresses the SSAP by proposing an integer programming formulation and a heuristic method to support the decision-making of large-scale instances. We introduce an Iterated Local Search (ILS) heuristic \citep{Lourenço2003}, an efficient metaheuristic applied to a broad range of optimization problems. Computational experiments on real and simulated datasets indicate a good performance of ILS speciallly on complex instances, where a high number of conflicts between students is observed.

The primary contributions of this paper are:
\begin{itemize} 
\item The investigation and formal definition of the Student Seat Assignment Problem (SSAP), aimed at optimizing the classroom environment for improved pedagogical outcomes; 
\item The proposal of a pedagogically-grounded mathematical model that enforces a strict priority among assignment criteria, effectively addressing the impact of row and column layouts on student supervision; 
\item The definition of upper and lower bounds, which were utilized to calibrate the parameters for both the heuristic and the mathematical model; 
\item The development of an Iterated Local Search (ILS) metaheuristic with tailored components specifically designed to navigate the hierarchical objective function and the hard constraints of the SSAP; 
\item A comparative analysis between the solutions generated by the proposed heuristic and existing real-world assignments curated by educators.
\end{itemize}

The rest of this paper is organized as follows. Section~\ref{sec:relatedworks} reviews recent relevant literature approaching the challenges of the seating arrangement, expanding to applications beyond the educational settings. Section~\ref{sec:salp} introduces the SSAP at length, introducing an integer programming formulation. Section~\ref{sec:ils} presents the proposed ILS, specially designed for the SSAP.  Section~\ref{sec:compexp} shows the computational experiments carried out in real and artificial datasets considering both ILS and a commercial solver.  Finally, Section~\ref{sec:conclusions} concludes this study by suggesting future research directions.


\section{Related Works}
\label{sec:relatedworks}

The SSAP is focused on optimizing the assignment of students to desks to minimize interpersonal conflicts in seating arrangements. This issue is relatively underexplored in existing research in the educational context \citep{Compani}. Therefore, this section provides a brief literature review of a seating plan challenges in diverse contexts, highlighting the most relevant studies related to this topic.

\subsection{Seating Arrangement Problems}

The seating assignment literature has historically been fragmented across isolated domains. A recent comprehensive review by \cite{Compani} identified that seating assignment models and methods have been explored across four main domains: transportation, dining, entertainment, and classrooms.

Various seating arrangement problems present challenges that are also relevant to the SSAP. For instance, from a game-theoretic perspective, \cite{10.5555/3398761.3398979} introduced a seating arrangement model focused on organizing agents according to their preferences for neighboring seats, highlighting the difficulty of achieving fair and stable plans. \cite{10.24963/ijcai.2023/285} investigated four NP-hard variants of the seating arrangement problem: Multi-Winner Assignment (MWA), Multi-Utility Assignment (MUA), Envy-Free Assignment (EFA), and Exchange-Stable Assignment (ESA).

\cite{Hill_2024} proposed the social classroom seating assignment problem, which aims to maximize new social connections among students by exposing them to peers they do not yet know. The problem is parameterized by students' existing social networks and the physical seating structure of the classroom, incorporating tie potentials that represent the likelihood of two neighbors forming a connection. The authors developed compact integer programming formulations and fast heuristics guided by network centrality measures, demonstrating applicability on realistic instances through a practical case study in which optimized arrangements led to substantial growth in new social interactions and favorable feedback from both instructors and students.

Several recent studies focus on spatial optimization for health safety. \cite{greenberg2021automated} modeled seat allocation in classrooms with fixed seats as a maximum independent set problem, with the goal of maximizing the number of occupied seats while ensuring that all assigned seats are at least a minimum safe distance apart, in accordance with COVID-19 social distancing guidelines. In the same context, \cite{Bayram2023} proposed optimization models and heuristic algorithms for seat allocation in classrooms to minimize infection risks through social distancing, highlighting the importance of adaptive strategies in fixed-seat settings such as theaters.

\cite{Karakose2024ARS} introduced the Maximum Diversity Social-Distancing Problem (MDSDP), a novel variant of the Maximum Diversity Problem applied to classroom seat allocation for public health safety. The authors proposed a two-phase optimization approach: the first phase applies a binary search algorithm to maximize the minimum distance between any pair of students, while the second phase employs mixed-integer programming to maximize the average distance among all occupants. Unlike prior models that rely solely on threshold-based safety limits, this approach ensures robust and uniform dispersion of students regardless of classroom layout, providing a mathematical framework to optimize room capacity while maintaining effective social distancing.

From a different perspective, \cite{Lewis2016} addressed seating arrangements for social events such as weddings, modeling the problem as a generalization of weighted graph coloring and $k$-partition problems, and employing a two-stage heuristic algorithm to produce high-quality solutions efficiently. \cite{VANGERVEN2022914} explored how seat distribution affects political dynamics in parliamentary settings, introducing a mixed-integer programming model and a heuristic that provide adaptable solutions for various chamber layouts. Finally, \cite{Shiina2023} addressed optimal seat allocation in the aviation industry, employing a stochastic programming model to maximize revenue while accounting for passenger preferences.

In the following section, we provide a detailed description of the SSAP, outlining the specific constraints that characterize the problem.

\section{The Student Seat Allocation Problem}
\label{sec:salp}

The Student Seat Allocation Problem (SSAP) is an assignment problem that involves optimizing the student-to-desk mapping in a classroom with a traditional layout, where seats are arranged in parallel rows, and the number of seats per row may vary. 

The main goal of the SSAP is to assign students to desks, using information about pairwise conflicts between students and other seating preferences, to ensure a smooth learning environment. In this problem, we assume that the number of available seats is sufficient to allocate all students. The problem adheres to the following conditions for allocating students in conflict:

\begin{enumerate}
    \renewcommand{\labelenumi}{\roman{enumi}.}
 \item Students with interpersonal conflicts should preferably be 
assigned either to seats in non-consecutive rows, or to seats within 
the same row with at least one seat between them;

\item If neither condition can be satisfied for some conflicting pairs and they must be assigned to consecutive rows, they should not be seated directly side-by-side across rows, nor in immediately adjacent front-to-back positions.
    \end{enumerate}

Additionally, individual priorities are taken into account when allocating students to the front or back rows of the classroom. 

\begin{mydef}
\justifying
Let $\Lambda$ be the set of rows in the classroom and $n_{\lambda}$ the number of seats (desks) in  row $\lambda$, $\lambda \in \Lambda$. Moreover, consider $D_{\lambda p}$ the seat in the \(\lambda\)-th row and \(p\)-th position, $p \in \{1, ... , n_{\lambda}\}$.
\label{def2.1}

\begin{enumerate}[(i)] 
    \justifying
    \item (Front seats) For each row $\lambda \in \Lambda$, the seats $D_{\lambda 1}$ and $D_{\lambda 2}$ are considered as ``in the front" of the classroom.\label{hyp1}
    \item (Back seats) For each row $\lambda \in \Lambda$, the seats $D_{\lambda n_{\lambda}}$ and $D_{\lambda (n_{\lambda}-1)}$ are considered as ``in the back" of the classroom.\label{hyp2}
\end{enumerate}

\end{mydef}

The concepts described in Definition~\ref{def2.1} are illustrated in Figure~\ref{fig:fig2.1}, where the educator's desk and other essential elements have been omitted to focus solely on the seats and rows. Figures~\ref{fig:fig2.1}(a) and~\ref{fig:fig2.1}(b) highlight in gray the desks considered in this problem as being in the front and back, respectively.

\begin{figure}[!htb]
\centering
\includegraphics[width=\textwidth]{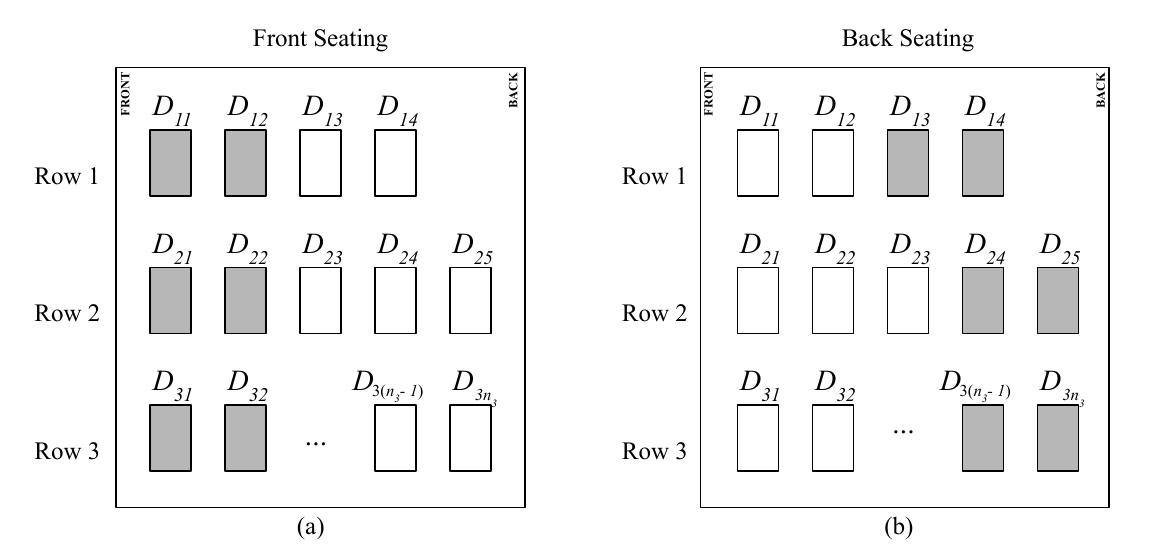}
\caption{A hypothetical representation of a traditional classroom with three parallel rows is shown. Rows 1, 2, and 3 contain 4, 5, and $n_3$ desks, respectively. The desks highlighted in gray represent seats at the front (a) or back (b) of the classroom.
} \label{fig:fig2.1}
\end{figure}

In  SSAP, the two seats at the front of each row are designated as the front seats, while the two seats at the back of each row are designated as the back seats. The following subsection presents the problem modeling.

\subsection{Mathematical Formulation}
\label{subsec:sample_a}

The mathematical model of the problem involves conceptualizing the row classroom as layers of a hierarchical graph. In this scenario, vertices (students) are placed at the open slots (empty desks) in the layers (or rows), and edges are considered actives when conflicting students are allocated in consecutive layers.

\begin{mydef}(Requirements for students)
Let \(\mathcal{S}\) denote a set of students in a classroom. For each student \(i \in \mathcal{S}\), the requirement \(r_i\), where \(r_i \in \{-1, 0, 1\}\), specifies the desired seating position for the student. Specifically:
\begin{itemize}
    \item \(r_i = -1\) indicates that the student needs to sit in the back of the classroom.
    \item \(r_i = 0\) indicates that there is no specific seating requirement for the student.
    \item \(r_i = 1\) indicates that the student needs to sit in the front of the classroom.
\end{itemize}
\label{def2.2}
\end{mydef}


\begin{mydef} 
Let $i, j \in \mathcal{S}$ be students in a classroom. Define the conflict indicator function \( c_{ij} \) as:
\[
c_{ij}  =
\begin{cases}
1, & \text{if the students i and j are in conflict} \\
0, & \text{otherwise}
\end{cases}
\]
\end{mydef}

 \begin{mydef}(Active edge) An edge is considered active if there are students in conflict in consecutive rows. Specifically, an edge $(i, j)$ in rows \(\lambda\) and \(\lambda+1\), where $\lambda, \lambda+1 \in \Lambda$, is active if $c_{ij} = 1$  and student $i$ (or $j$) is allocated to a seat in row $\lambda$ and student $j$ (or $i$) is assigned to a seat in row $\lambda + 1$. Adjacent rows to $\lambda$ are rows $\lambda - 1$ and $\lambda + 1$. 
\label{def2.3}
\end{mydef}

Moreover, the distance between a pair of students assigned to 
consecutive rows --- one seated at position $k$ in row $\lambda$ 
and the other at position $z$ in row $\lambda+1$ --- is defined 
as $|z - k|$, consistently with how distance is measured between 
students within the same row.

The SSAP utilizes a hierarchical objective function that enforces a strict priority among assignment criteria. These criteria are grounded in pedagogical experience and qualitative feedback from educators, reflecting the fact that row and column configurations have distinct impacts on classroom supervision and student interaction.
Teachers’ experiences indicate that assigning conflicting students to non-consecutive rows or the same row is preferable to placing them in consecutive rows, as the latter configuration more easily facilitates disruptive interaction in typical classroom layouts.
In scenarios where students are placed in the same or consecutive rows, the model enforces a minimum separation distance. These preferences are further influenced by visibility and supervision factors that standard distance measures may not fully capture. When a consecutive-row allocation is unavoidable, the objective function seeks to maximize the spacing between students while strictly maintaining the required minimum distance.

The decision variables, parameters, and integer programming model are presented next.

\begin{center}
\begin{longtable}{ccp{10cm}}
  Parameters & & \\
  $i, j$ &: & indexes representing students (vertices);\\
  $\lambda$ &: & index representing rows (layers);\\
  $k, z$ &: & indexes representing positions (seats/desks);\\
  $S$ & : & set of students in a classroom; \\
$|S|$&: & total number of students;\\
  $\Lambda$ & : & set of rows in a classroom;  \\
  $|\Lambda|$ & : & total number of rows; \\   
  $n_{\lambda}$ & : & number of desks in row $\lambda \in \Lambda$;\\
  $c_{ij}$ & : & parameter that receives value 1 if student $i$ has a conflict (edge) with student $j$, and 0, otherwise;\\
  $E$ & : & set of edges  $(i,j)$ from the conflict graph, where $i<j$ and $c_{ij}=1$;\\
  $|E|$ & : & total number of edges from the conflict graph;\\
  $r_i$ &: & parameter that takes the value 1 if student $i$ needs to sit in the front of the classroom, -1 if student $i$ needs to sit in the back of the classroom, and 0 if there are no specific seating requests for student $i$;\\
  $d_{min}$ &: &  defines the minimum distance required between conflicting students in consecutive rows; \\
 $d_{min}'$ &: &defines the minimum distance required between conflicting students within the same row;  \\
  $\Psi$ & : & active edge weight;\\
  $\mathop{n} $ & : & the total number of desks. \\\\  
  Decision variables &  & \\
  $x^{\lambda}_{ik}$ & : & binary variable that receives value 1 if student $i$ is in row $\lambda$ at $k$-th position, and 0 otherwise;\\
  $\mathop{w^{\lambda}_{ijkz}}$ & : & auxiliary binary variable that takes the value 1 if students $i$ and $j$ are in layers $\lambda$ and $\lambda + 1$, respectively, in positions $k$ and $z$, in that order, and 0 otherwise.\\
 \end{longtable}
\end{center}

The SSAP model is formulated as follows:

\allowdisplaybreaks
\begin{align}
    \max f(x,w,y): \sum_{\lambda=1}^{|\Lambda|-1} \sum_{k=1}^{n_{\lambda}} \sum_{z={1}}^{n_{(\lambda + 1)}} \sum_{\substack{ (i,j)\in E} }\left( \mathop{|z - k| -   \Psi} \right) \cdot w^{\lambda}_{i j k z}   \label{obj1}
\end{align}

\noindent subject to:

\begin{align}
    &\sum_{\lambda=1}^{|\Lambda|} \sum_{k=1}^{n_{\lambda}} x^{\lambda}_{ik} = 1, \quad \forall i \in S \label{r2} \\
    &\sum_{i=1}^{|S|} x^{\lambda}_{ik} \leq 1, \quad \forall \lambda \in \Lambda, \, k \in [1, n_{\lambda}] \label{r3} \\
    & z - k \geq (x^{\lambda}_{ik} + x^{\lambda}_{jz} - 1) \cdot d_{min}', \quad \forall (i,j) \in E, \lambda \in \Lambda, k \in [1,n_{\lambda}-1], z \in [k+1,n_{\lambda}] 
    \label{r5}\\
    &w^{\lambda}_{i j k z} \geq x^{\lambda}_{ik} + x^{\lambda+1}_{jz} - 1, \quad \forall (i,j) \in E, \lambda \in [1,|\Lambda|-1], k \in [1,n_{\lambda}], z \in [1,n_{\lambda+1}] 
    \label{r7} \\
    &w^{\lambda}_{i j k z} \leq x^{\lambda}_{ik}, \quad \forall (i,j) \in E, \lambda \in [1,|\Lambda|-1], k \in [1,n_{\lambda}], z \in [1,n_{\lambda+1}]
    \label{r8} \\
    &w^{\lambda}_{i j k z} \leq x^{\lambda+1}_{jz}, \quad \forall (i,j) \in E, \lambda \in [1,|\Lambda|-1], k \in [1,n_{\lambda}], z \in [1,n_{\lambda+1}]
    \label{r9} \\
    &|z - k|\cdot w^{\lambda}_{i j k z}\geq d_{min}\cdot w^{\lambda}_{i j k z}\quad \forall (i,j) \in E, \lambda \in [1,|\Lambda|-1], k \in [1,n_{\lambda}], z \in [1,n_{\lambda+1}]
    \label{r10} \\ 
    &|z - k|\cdot w^{\lambda}_{i j k z} \leq \max(n_{\lambda} - 1, n_{(\lambda+1)} - 1),  \forall \lambda \in [1,|\Lambda|-1], k \in [1,n_{\lambda}], z \in [1,n_{\lambda+1}], (i,j)\in E
    \label{r11}\\
    & \sum_{\lambda=1}^{|\Lambda|} x^{\lambda}_{i 1} + \sum_{\lambda=1}^{|\Lambda|} x^{\lambda}_{i 2} = 1 \quad, \forall i \in S, \quad r_i = 1 \label{r12} \\
    & \sum_{\lambda=1}^{|\Lambda|} x^{\lambda}_{i n_{\lambda}} + \sum_{\lambda=1}^{|\Lambda|} x^{\lambda}_{i (n_{\lambda}-1)} = 1 \quad, \forall i \in S, \quad r_i = -1 \label{r13}
\end{align}


The objective function \eqref{obj1} aims to minimize the number of active edges while maximizing the distance between conflicting students across consecutive layers. The weight $\Psi$ emphasizes minimizing the number of active edges in the objective function. It ensures that the penalty for active edges significantly influences the optimization, prioritizing their reduction when maximizing the distance between them. A high value of $\Psi$, increases the penalty for solutions with  active edges, promoting their minimization over other considerations in the objective function. Moreover, the nonlinear absolute function  can be replaced by breaking the sum considering the cases when $z>k$ and $z<k$. 
 
To preserve the hierarchical behavior of the objective function, the weight parameter $\Psi$ must satisfy a sufficiently large lower bound. Since the maximum horizontal separation between two seats is bounded by $\max_{\lambda}(n_{\lambda}-1)$, the parameter must satisfy $\Psi > \max_{\lambda}(n_{\lambda}-1)$. This condition ensures that any reduction in the number of active edges always dominates any possible gain obtained from the distance maximization term.

Constraints \eqref{r2} ensure that each student $i$ occupies exactly one desk. Constraints  \eqref{r3} ensure that a desk is assigned to at most one student. 
Constraints \eqref{r5} ensure that conflicting students in the same layer have at least $d_{min}'-1$ desks between them. Constraints \eqref{r7} to \eqref{r9} ensure that $w^{\lambda}_{i j k z}$ is 1 when $c_{ij}$ is 1 and student $i$ is in layer $\lambda$ at position $k$ and, student $j$ is in the consecutive layer $\lambda + 1$ at position $z$. Otherwise, $w^{\lambda}_{i j k z}$ is null. Note that $w^{\lambda}_{i j k z}$ is an auxiliary binary decision variable used to linearize the expression $x^{\lambda}_{ik} \cdot x^{\lambda+1}_{jz}$.  The value of $d_{min}'$ and $d_{min}$ must be greater than or equal to 2 to ensure at least one desk between two conflicting students. Additionally, their values should not exceed \(\min_{\lambda \in \Lambda} \{n_{\lambda}\}\). For ease of notation, we treat $d_{min}'$ as equivalent to $d_{min}$ moving forward, reflecting the standard application of this parameter in the current study. Constraints \eqref{r10} and \eqref{r11} establish the lower and upper limits for the distance between conflicting students allocated in consecutive rows. The nonlinear absolute function  can be replaced by considering the constraints when $z>k$ and $z<k$, as discussed for the objective function. Constraints \eqref{r12} and \eqref{r13} ensure that students' requirements are met, placing them in the front seats (first or second seats) or in the back seats (penultimate or last seats), respectively.

Regarding the dimensionality of the formulation, let \(\bar{n} = \max_{\lambda \in \Lambda} n_{\lambda}\) be the maximum number of desks in a row. One can observe that the formulation contains $O(n \cdot |\Lambda| \cdot \bar{n})$ binary variables $x^\lambda_{ik}$, and $O(|E| \cdot |\Lambda| \cdot \bar{n}^2)$ auxiliary binary variables $w^\lambda_{ijkz}$. The latter term dominates the total variable count.

Constraints~(6)--(10) dominates in count possessing a \(O(|E||\Lambda|\bar{n}^2)\) number of constraints. The remaining assignment and requirement constraints, such as Constraints~(2), (3), (11), and (12), contribute lower-order terms to the constraint count.

Real-world educational networks, including those characterized by severe behavioral challenges and high friction rates, consistently preserve a significant degree of topological sparsity relative to a complete graph ($|E| \ll n^2$).

The proposed SSAP is NP-hard, as better discussed in \ref{complexity}.

\section{Iterated Local Search for the SSAP}
\label{sec:ils}

In this paper, we propose an Iterated Local Search (ILS) \citep{Lourenço2003} to heuristically solve the Student Seat Allocation problem. Algorithm \ref{alg:alg1}  presents a high-level ILS pseudocode. 

\textbf{\begin{algorithm}[!htb]
\caption{Iterated Local Search (ILS)}
$s_0$ = Initial Solution()\;
$s^*$, $s$ = Local Search($s_0$)\;
\While{stopping condition has not been satisfied}{
    $s'$ = Perturbation($s$)\;
    $s'$ = Local Search($s'$)\;
    $s$=Acceptance Criterion ($s', s^*$)\;
        \If{$s'$ is better than  $s^*$}{
         $s^*$ = $s'$\;
    }
}
\Return $s^*$
\label{alg:alg1}
\end{algorithm}}

ILS is a metaheuristic that enhances the search process by iteratively searching for solutions through perturbations and local search. As can be observed in Algorithm~\ref{alg:alg1}, the heuristic begins by generating an initial solution and then applying a local search to reach a local optimal solution. This local optimum is referred to as the reference solution. A perturbation is introduced to the reference solution, which modifies the current solution. A local search is then applied to the perturbed solution, producing a new local optimum. The algorithm employs an acceptance criterion to decide whether to replace the reference solution with the one found in the iteration. The stopping condition of the routine can be achieving a given maximum number of iterations. The best solution found during the process is then returned as the final solution.
ILS was adopted in this paper due to its consistently strong performance in tackling challenging combinatorial optimization problems. ILS is widely recognized in the literature for being a powerful heuristic for optimization problems, including vehicle routing problems \citep{Maximo2022,Maximo2024}, location problems \citep{Maximo2025}, graph coloring problems \citep{NOGUEIRA2021105087}, university course timetabling \citep{SONG2018597}, assignment problems \citep{LI2018120}, and others. Its approach allows for exploring infeasible solutions, facilitating the escape of local optima, which is crucial in complex contexts. The simplicity and modularity make ILS implementation intuitive, while the capability to incorporate new local heuristics and perform stochastic and directed searches contributes to its success in combinatorial optimization scenarios \citep{Lourenço2003, CARAMIA2008201}.

The following sections present the algorithms that compose the introduced ILS.

\subsection{Initial solution}
\label{subsubsec:sample_b1}

This section presents the algorithm to find an initial solution to the introduced ILS.

\subsubsection{Weight Matrix Construction}
Let $D_{\lambda,p}$ be the desk $p$-th located at the $\lambda$-th layer. To generate an initial solution, we first construct a square matrix  \(\mathbf{A}\) of order $n$, where $n$ is the number of seats, referred to as the weight matrix. Each row in \(\mathbf{A}\) corresponds to a student whereas each column represents a desk in the classroom. In line with this, the column of \(\mathbf{A}\) that corresponds to the first desk of layer $\lambda$ is:

\begin{equation}f_{\lambda}= 1 + \sum_{\lambda'=1}^{\lambda-1} n_{\lambda'} \end{equation} \label{mapping}

Therefore,  the column of \(\mathbf{A}\) that corresponds to desk $D_{\lambda,p}$ is $f_{\lambda}+(p-1)$. For example, in Figure \(\ref{fig:fig2.2}\), desk 7 (the seventh column of  \(\mathbf{A}\)) is the third seat in the second layer of the classroom, position $(2,3)$. 
\begin{figure}[!htb]
\centering
\includegraphics[width=\textwidth]{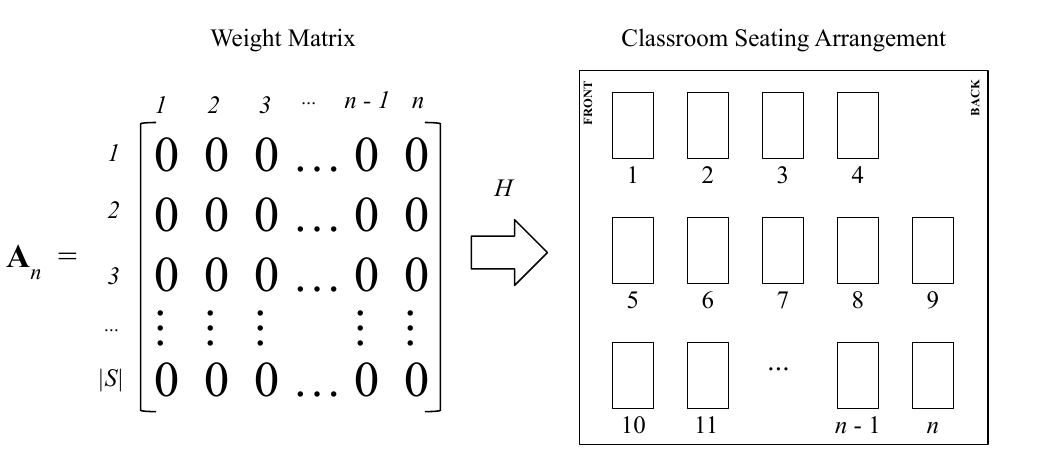}
\caption{Representation of the initially zero-weight matrix \(\mathbf{A}\) and the correspondence of its columns to the traditional classroom layout.} \label{fig:fig2.2}
\end{figure}

The last seat of a layer $\lambda$ is:
\begin{equation}b_{\lambda}=\sum_{\lambda'=1}^{\lambda}  n_{\lambda'}\end{equation}\label{mapping2}

The distance between a pair of desks in the same or consecutive layers is the absolute difference between their $p$ values. For example, in Figure \ref{fig:fig2.2}, the distance between desks $D_{2, 2}$ and $D_{1, 4}$ is $|2-4|=2$. 

The weight matrix undergoes an iterative process that updates its values, with null values representing its initial state. The idea is to simulate a positioning of conflicting students to fill the matrix with values 
that signalize prohibitive assignments. 
Let $G=(V, E)$ be a conflict graph. The set of vertices $V$ represents the students involved in some conflict, and $E$ is composed of all conflicting edges.

Bearing in mind that the column in \(\mathbf{A}\) that represents desk $D_{\lambda, p}$ is $l=f_{\lambda}+(p-1)$, for each conflicting student \(j \in \mathcal{N}(i)\), where $\mathcal{N}(i)$ is the set of neighbors of $i$, positive weights are assigned to $a_{j,l'}$ if $l'$ corresponds to a desk too close to \(l\), as defined by the problem constraints. This process repeats for all students in $G$. 

Given that the problem constraints require a minimum distance between conflicting students in the same and adjacent rows, the values assigned to matrix \(\mathbf{A}\) depend on the distance between students' desks. To better reflect these constraints, the attributed values vary in the interval [$-n\Delta$, $2\cdot\Delta$], where $\Delta = \max_{\lambda \in \Lambda} \{n_{\lambda}\} + 1$.
This value for $\Delta$ was adopted to ensure that matrix elements corresponding to undesirable assignments are significantly greater than zero. Penalties are applied and updated in the weight matrix, where the rows correspond to the students in conflict, and the columns correspond to the undesirable desks. Besides, the matrix elements indicating the pair student-desk where the student needs to occupy back or front seats must receive very low (negative) penalties as initial values to stimulate these assignments. 

The updating of the weight matrix is outlined in the following steps (see \textit{Weight Matrix Update}). 
{\small
\vspace{0.5cm}
\hrule
\vspace{0.05cm}
\noindent\textit{Weight Matrix Update} 
\vspace{0.05cm}
\hrule
\begin{enumerate}[\textit{Step} 1]

\item Let \( \overline{S} \) be the set of students who have at least one conflict, meaning that they belong to the conflict graph. Students are ordered using a composite rule: (i) students with higher conflict degree come first; (ii) among students with equal degree, those with front or back seating preferences are prioritized over neutral students; (iii) in the event of equal priority and degree, ties are broken randomly.

\item Initialize the weight matrix $\mathbf{A}$, inserting negative weights into the matrix positions where the rows correspond to students who should sit at the front or back, and the columns correspond to desks located at the front or back, respectively:
\[
a_{il} =
\begin{cases}
- n \cdot \Delta, & \forall i \in \overline{S} \land (r_{i} = 1),\ l = f_{\lambda}, f_{\lambda} - 1 ,\ \forall \lambda \in \Lambda, \\
- n \cdot \Delta, & \forall i \in \overline{S} \land (r_{i} = -1),\ l = b_{\lambda}, b_{\lambda} - 1,\ \forall \lambda \in \Lambda, \\
0, & \text{otherwise}.
\end{cases}
\]

 The simulation process involves the following steps, where students are hypothetically assigned to desks to update the elements of \(\mathbf{A}\). Let $i$ be the first element (student) from  \(\overline{S}\). 
 
\item Assign student \(i\) to a random seat \(D_{\lambda,p}\) not selected in previous steps, adhering to the students' preferences (front and back desk constraints) and ensuring non-overlapping assignments of student-desk pairs.\label{stepbefore}

\item Let $w$ be the column in \(\mathbf{A}\) that indicates  seat \(D_{\lambda,p}\)\footnote{According to Equation~\eqref{mapping}, $w=f_{\lambda}+ (p - 1)$.}. The elements of row \(i\) of the matrix \(\mathbf{A}\) are incremented by \(2 \cdot \Delta\), except for column \(w\) to ensure that the position \((i, w)\) has the lowest value. Similarly,  all elements in the column $w$ except for row \(i\) are incremented by \(2 \cdot \Delta\) to ensure that student \(i\) is the preferred choice for the seat at position \(w\). The columns of \(\mathbf{A}\) that indicate desks in the layer $\lambda$ with distance higher than
($d_{min}-1$) to $D_{\lambda,p}$ are referred to $\mathcal{N}_w$\footnote{This set is composed by columns $f_{\lambda}, f_{\lambda}+1, \ldots, w-2, w+2,\ldots, b_{\lambda}$}.  We say a pair of desks is adjacent if the distance between them is lower than or equal to 1. 

The set of columns of \(\mathbf{A}\) that indicate desks from layers $\lambda-1$ and $\lambda+1$ adjacent to $D_{\lambda,p}$ is referred to as $\mathcal{N}_w^1$. The set of columns of \(\mathbf{A}\) that indicate desks from layer $\lambda$ with distance less than or equal to $d_{min}-1$ from $D_{\lambda,p}$ is referred to as $\mathcal{N}_w^2$. The set of columns of \(\mathbf{A}\) indicating desks from layers $\lambda-1$ and $\lambda+1$ with distance higher than 1 to $D_{\lambda,p}$ is referred to $\mathcal{N}_w^+$. The process of updating  \(\mathbf{A}\) having student $i$ and desk $D_{\lambda,p}$ as pivot works as follows:
\[
a_{j, w'} :=
\begin{cases}
a_{j, w'} + 2 \cdot \Delta, & \text{where } j \in \mathcal{N}(i) \text{ and } w' \in \mathcal{N}_w^1 \cup \mathcal{N}_w^2, \\
a_{j, w'} + 2 \cdot \Delta - 0.1 \cdot |w - w'|, & \text{where } j \in \mathcal{N}(i) \text{ and } w' \in \mathcal{N}_w^+, \\
a_{j, w'} + \Delta - 0.1 \cdot |w - w'|, & \text{where } j \in \mathcal{N}(i) \text{ and } w' \in \mathcal{N}_w.
\end{cases}
\]

Note that the element value is reduced as the distance between desks in the same (or consecutive) layer increases. Moreover, the constant multiplying $\Delta$ expresses the preference for placing conflicting students in the same layer instead of consecutive layers.

\item Update the investigated student $i$ by picking the next in the ordered set \( \overline{S} \). \label{stepupdatestudent}

\item Go to \textit{Step} \ref{stepbefore} if there is at least one  student in \( \overline{S} \) not assigned to a desk in the simulation process. Otherwise, return matrix \(\mathbf{A}\). \label{laststep}
\vspace{0.5cm}
\hrule
\vspace{0.05cm}
\end{enumerate}
}


\subsubsection{Solution Construction}

Before discussing the routine for constructing a solution to the SSAP from matrix \(\mathbf{A}\), let us define a penalized objective function to evaluate infeasible solutions.

Let $P=\{x,w,y\}$ be an infeasible solution for the SSAP where: 
\begin{itemize}
    \item \( \alpha \) represents the total number of students who should be seated in the front of the classroom but are not,
    \item \( \beta \) represents the number of students who should be seated in the back of the classroom but are not,
    \item \( \gamma \) represents the total number of arcs (students in conflict) in the same layer that are less than $d_{min}$ units apart,  and
    \item \( \delta \) represents the total number of active arcs (students in conflict on consecutive layers) that are less than two units apart. 
\end{itemize}

Definition~\ref{def4} presents the penalized objective function used in the proposed heuristic.
\begin{mydef}(Penalized Objective Function) Let $P$ be an infeasible solution for the SSAP. The penalized objective function used in the heuristic, which allows for the evaluation of incomplete solutions and solutions with violations of positioning constraints, is defined as follows:

\begin{equation} 
f_p(P)= f(P)-\phi \cdot (\alpha + \beta + \gamma + \delta)
\label{fo} 
\end{equation}
\noindent where $\phi>0$ is the penalty factor.
\label{def4}
\end{mydef}

The construction of an initial solution for the introduced ILS proceeds through the following primary steps.
{\small
\vspace{0.5cm}
\hrule
\vspace{0.05cm}
\noindent\textit{Initial Solution Construction} 
\vspace{0.05cm}
\hrule
\begin{enumerate}[\textit{Step} 1]

\item Shifting normalization of matrix $\mathbf{A}$ to ensure non-negative values by subtracting the matrix's minimum value from each element:
\[
\mathbf{A} \leftarrow \mathbf{A} - \min(\mathbf{A})
\]
Let $D'$ be the set of tuples $(\lambda, p)$ indicating all desks $D_{\lambda,p}$.
\item Construct a Partial Solution (\(P=\{x,w,y\}\)) considering only the conflicting students prioritized for assignment, starting with the one with the highest degree of conflict: 
\begin{enumerate}[(i)]

\item  Let $D=\{{(\lambda,p)}\in D': \nexists j \in \mathcal{N}(i) | (x_{jp}^{\lambda}=1 \lor x_{jp}^{\lambda-1}=1\lor x_{jp}^{\lambda+1}=1)  \} $. If $D= \emptyset$ then $D=D^1\cup D^2$, where $D^1=\{(\lambda,p)\in D': j \in \mathcal{N}(i) | x_{jp}^{\lambda}=1 \land d_{ij}> d_{min} - 1  \} $ and $D^2=\{(\lambda,p)\in D': j \in \mathcal{N}(i) | x_{jp}^{\lambda-1}=1 \land d_{ij}>1 \lor  x_{jp}^{\lambda+1}=1 \land d_{ij}>1  \} $
\item Assign student \(i\) to a desk if $D\neq \emptyset$:
\\$x_{ik}^{\lambda^*}=1$, where $(\lambda^*,k)=\arg \min_{(\lambda, p)\in D}\{a_{i,t}: t=f_{\lambda}+p-1\}$;
\item Update $\mathbf{A}$ and $D'$:\\
$a_{i,t}= a_{i,t}+\theta, \forall j\in \mathcal{N}(i), t=f_{\lambda}+r$, where $(\lambda,r)\in \mathcal{N}^1(\lambda^*,k)$\footnote{The 1-neighborhood $\mathcal{N}^1$ of $(\lambda,r)$ is: $\{(\lambda^*,k-1), (\lambda^*,k+1), (\lambda^*-1,k-1),(\lambda^*-1,k),(\lambda^*-1,k+1), (\lambda^*+1,k-1),(\lambda^*+1,k),(\lambda^*+1,k+1)\}$.} and $\theta$  is a positive scalar;\\
$D'=D'\backslash (\lambda^*, k)$

\end{enumerate}

\item 
Compute the penalized objective function based on the partial solution \( P \) as defined in Equation~\eqref{fo};

\item Search for improvements through swaps: starting from the partial solution, the assigned students are swapped among the unoccupied seats to check for any improvements in the initial penalized objective function value. If a movement provides an improvement, the solution is updated;

\item After optimizing the partial solution, the unallocated students with a degree greater than zero and those with a degree of zero who have some seat preference will be inserted while respecting the problem's constraints;\label{stepinitial}

\item Randomly insert the remaining students (including those with a degree of zero): the students who have not yet been allocated to a desk are assigned to the remaining available seats, resulting in the initial solution (\( I \)).

\end{enumerate}
\hrule
\vspace{0.5cm}
}
Following the weight matrix update, \(\mathbf{A}\) undergoes shifting normalization, from which a partial solution $P$ is derived. In this phase, students are assigned to classroom seats based on the information contained in the weight matrix \(\mathbf{A}\) while respecting conflict constraints between students. The weight matrix defines the ``cost'' of placing a student in a specific position (see Figure \ref{fig:fig2.3}). 
Based on the obtained partial solution, the initial solution is created by making adjustments to improve the values of the penalized objective function, incorporating some elements of randomness. Note that in \textit{Step} \ref{stepinitial}, starting from a feasible partial solution of the problem, we attempt to insert the conflicting students not yet allocated to a desk into the available empty positions. In this case, the method investigates each unassigned conflicting student, trying to place them into available empty positions in the classroom, one at a time. If the insertion is successful, the student can sit at a desk while adhering to the constraints and feasibility of the solution. Therefore, the partial solution is updated, and we examine the next student. Finally, the updated partial solution and the unplaced students are returned. These students are then randomly inserted, increasing the likelihood of infeasibility in the initial solution.

\begin{figure}[!ht]
\centering
\includegraphics[width=\textwidth]{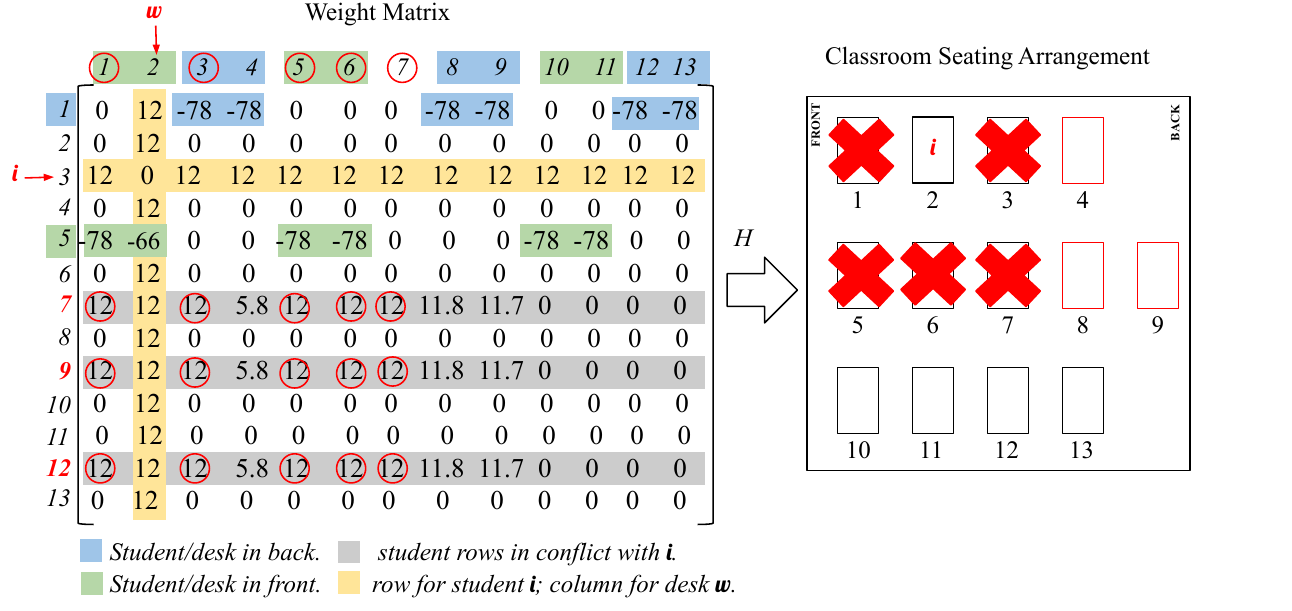}
\caption{Example of an updating of matrix \(\mathbf{A}\) based on unwanted positions. Desk \(w = 2\) should be reserved for student \(3\) and marked as the best option. Add \(2 \cdot \Delta\), where $\Delta =6$, to all other columns (desks) in row \(i\) and to all other rows (students) in column (desk) \(w\) to ensure \(w\) is preferred to student \(i\) and unavailable to others. In the classroom layout, desks adjacent to \(w\) in row 1 (desks 1 and 3) and desks 5, 6 and 7 in the next row should be avoided by students in conflict with student \(i\) -- namely students 7, 9 and 13 --, ensuring at least one desk between them (\(d_{min} = 2\)). These positions receive the value \(12\). The matrix elements corresponding to desks in the neighboring layers, seats 7 and 8, have a slight discount in the 12 value. Seat 4, which is in the same layer but is distant by at least 1 seat, receives the value $\Delta -0.1\cdot 2$ in the positions corresponding to the conflicting students.}
\label{fig:fig2.3}
\end{figure}

The initial solution goes through a local search phase aiming to find a solution with a better penalized objective function value. Then we perturb the current solution, regardless of its feasibility  yielding a feasible or an infeasible solution. The method returns the best solution found when it reaches the stopping criterion. The best solution is then further improved by a refinement strategy if there exists any active edge in the final solution.

The following sections present the details of each of these steps.

\subsubsection{Neighborhood Movements}
\label{subsubsec:sample_b2}

This section presents a local search algorithm consisting of four distinct movements. It aims to find the optimal solution that maximizes the penalized objective function, as defined earlier.  The pseudocode descriptions of these movements can be found in Section 3 of the Supplementary Material.

\begin{itemize}
    \item \textbf{I - Swap I}: This strategy aims to escape local optima by introducing controlled randomness in the selection of students for potential repositioning. In each swap, approximately {$\psi$}\% of students in each row are randomly selected as diversification targets. For each of these students, a set of candidate swap positions is generated based on their seating preference: students with front or back preferences are restricted to positions in the front or back sections of the classroom, respectively; students with neutral preference consider all other positions in their current row and in two randomly selected rows. {To speed up the search process, we limited the number of positions evaluated to $\Gamma$\% of the possible positions for students with neutral preferences. O}nly improving swaps are accepted. The process repeats while improvements are found. This method balances diversification and efficiency by limiting both the number of students and swap evaluations. 
    \item \textbf{II - Swap II}: This movement exchanges students who are supposed to sit at the back of the classroom with those who are currently seated there.
    \item  \textbf{III - Swap III}: This activity switches students who are supposed to sit in the front of the classroom with those who are seated there.
    Its pseudocode is the same as Algorithm 2 (described in the Supplementary Material),
    but considering the positioning difference (instead of the back seat is the front seat). 

 \item  \textbf{IV - Swap IV}: This movement identifies active edges in the current solution, where students in conflict are seated in consecutive rows besides violating the minimum distance constraint. The method involves exchanging students in conflict with other students in the classroom. 
\end{itemize}

These movements are applied in this sequence once, where each movement is executed until it reaches a stopping criterion.

\subsubsection{Perturbation Phase}
\label{subsubsec:sample_b3}

The perturbation phase is part of the ILS metaheuristic chosen to solve the SSAP. The main goal of applying perturbation to the solution is to introduce randomness and better explore the search space. The routine modifies a solution by performing random swaps considering viable movements on a certain percentage of student-desk allocations. Viable movements refer to swaps that do not violate any already met constraint. This means that it is allowed a swap involving a student that violates a constraint in the original position and is not necessarily met in the new position. The input for this routine consists of the reference solution and the perturbation degree $\theta$, which is defined as the fraction of elements to be perturbed, where $\theta \in (0,1]$.

{\small
\vspace{0.5cm}
\hrule
\vspace{0.05cm}
\noindent\textit{Perturbation} 
\vspace{0.05cm}
\hrule
\begin{enumerate}[\textit{Step} 1]

  \item Calculate the perturbation degree:
  $\rho$ = $\lceil n \cdot\theta \rceil$ is
  the number of elements to swap, calculated by multiplying the total number of students (desks) by the selected percentage. The  function $\max$ ensures that at least one element is swapped;

 \item Select randomly students to swap in the reference solution;

\item Perform the swap of each selected student by randomly selecting a viable position for that student;

\end{enumerate}
\hrule
\vspace{0.5cm}
}

\subsubsection{Local Search}
\label{subsubsec:sample_b4}

As the proposed ILS allows infeasible solutions for the SSAP, the best overall solution is not necessarily feasible. Therefore, to minimize active edges and maximize the distance between conflicting students in consecutive layers, a refinement procedure is applied to the final solution. This procedure is explained next.

{\small
\vspace{0.5cm}
\hrule
\vspace{0.05cm}
\noindent\textit{Local Search} 
\vspace{0.05cm}
\hrule
\begin{enumerate}[\textit{Step} 1]

  \item { Minimizing Active Edges:}
\begin{enumerate}[(i)]
   
   \item Let \( L \) be the list of students who are the end nodes of active edges;

   \item Select a student \( i \) from \( L \); \label{active}

    \item  Let \( \bar{L}_{neighbors} \) be the set of layers where the neighbors of node \( i \) are located; 

   \item Let \( \bar{L} \) be the rows that are non-adjacent to the rows in \( \bar{L}_{neighbors} \); 

   \item If \( \bar{L} \neq \emptyset \), attempt to insert \( i \) into each of the other rows in \( \bar{L} \). Specifically, given the current position of \( i \) in its row, for each row in \( \bar{L} \), swap the position of \( i \) with the seat of the student in row \( \bar{L} \) that provides the best improvement in the current best solution; 

   \item Update \( L \) by removing students for whom the conflict was resolved and adding  end nodes of new conflicting edges; 
   \item If \( L \neq \emptyset \) or no further moves can eliminate active edges, go to Step (iii). 

\end{enumerate}

  \item Maximize distance 
       \begin{enumerate}[(i)]
        \item List the active edges and identify the rows containing the students involved.
        \item For each active edge ($i, j$), calculate the distance between the initial positions of students $i$ and $j$ and the objective function, which are considered the current best values.
        \item For each active edge, iterate over all possible positions in the row of student $i$ and, for each of these, over all positions in the row of student $j$. For each combination, recalculate the distance and objective function. If the new distance is greater than the previous one and the objective function is the same or improved, store this configuration as the best solution found so far.
        \item Apply this procedure to all active edges, prioritizing configurations with greater distances and equal or better objective function values.
        \end{enumerate}

\end{enumerate}
\hrule
\vspace{0.5cm}
}

Figure \ref{fig:fig3} presents the flowchart of the proposed ILS for the SSAP. The process begins with an initial solution $s_{0}$. A preliminary local search (Swap~I) then generates the  solution $s^{*}$, which is also used as the reference solution $s^{r}$. The search is guided by a penalized objective function $f(\cdot)$ that allows handling infeasible solutions. Each iteration of the algorithm: (i) obtains a perturbed solution $s$ by applying viable random swaps with degree $\theta$ to the reference $s^{r}$; (ii) runs a local search that sequentially applies neighborhoods Swap~I--IV to reach a local optimum $s'$; and (iii) accepts $s'$ if its penalized objective value $f(s')$ improves the current best solution $s^{*}$. The loop terminates when iteration or stagnation limits are reached, or upon achieving a target solution quality. Finally, a refinement step is applied to $s^{*}$ to eliminate active edges and, if conflicts still occur in consecutive rows, increase inter-student distances without worsening $f(\cdot)$. 

%
\begin{figure}[!htb]
\centering
\includegraphics[width=0.89\textwidth]{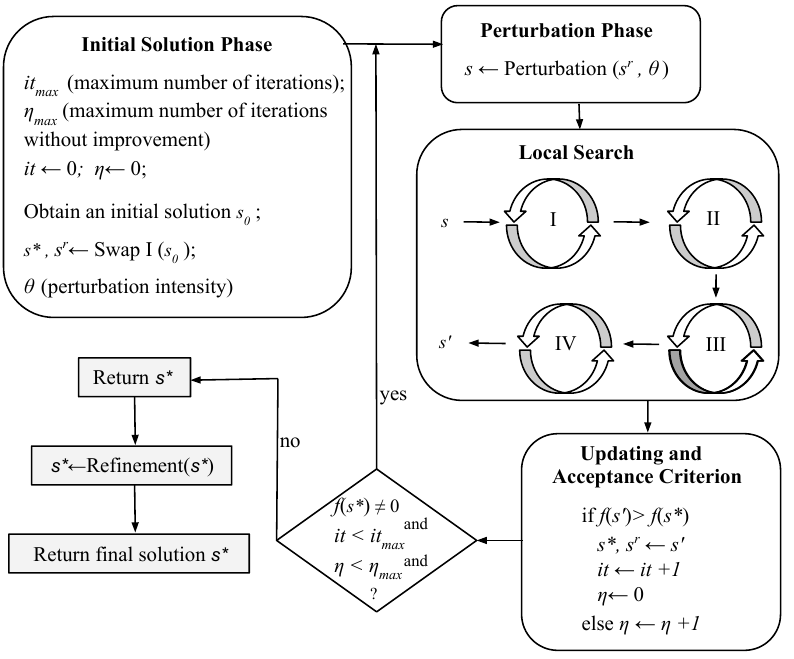}
\caption{The proposed Iterated Local Search to solve the SSAP.}
\label{fig:fig3}
\end{figure}

It is worth noting that the heuristic assumes a full classroom, where the number of students equals the number of available desks. However, it straightforwardly extends to settings where the number of students is smaller than the number of desks by introducing artificial conflict-free students to fill the remaining seats.

\section{Computational Experiments}
\label{sec:compexp}

This section presents two experiments conducted for the analysis of the proposed problem, considering the optimization solver Gurobi 10.0.3 and the introduced heuristic. As the heuristic is non-deterministic, we ran a number of independent executions to show the robustness of the proposed solution method and report the average values (gap and time). Moreover, we report the percentage of feasible solutions found in the independent executions and the average success rate.

The experiments employed two sets of instances. The first set is based on real data of a public basic education school in São José dos Campos, Brazil, consisting of graphs obtained from information collected from three junior high school classrooms, in Brazil known as the fundamental II level. These data were provided by teachers and identified through practical observation. The second set consists of 131 undirected simple graphs generated using a model based on the total number of nodes and edges for graph generation. In this case, different instances were generated by varying the total number of nodes and edges according to predefined percentages based on analyses of the real data set. Here, the balance of total students in the classroom, the proportion of nodes with conflicts, the proportion of conflicts, and the proportion of students with seating preferences were considered to define the percentage ranges for graph creation. The graphs generated in this process presented equivalent or greater complexity than the real data, allowing for a more thorough exploration of the proposed methods. Based on real-data collected from the schools, the number of available desks is exactly equal to the total number of students. Therefore, we assume $n=|S|$ in both experiments.

The first experiment employs the set of artificial instances. In this case, we assess the performance of the heuristic, contrasting its results with those obtained by Gurobi time limited in 3,600 seconds. The second experiment presents the results of the heuristic applied to the real case instances. Furthermore, an analysis of the initial solution's performance is presented in  \ref{subsec:sample_ini}.

The next section provides the details of the  instances used in the experiments.
\subsection{Real Instances}
\label{subsec:4.1}

The dataset used to generate the instances includes information from three distinct classrooms in the final years of a public basic education school in Brazil, referred to  as Classroom I, Classroom II, and Classroom III. Tables \ref{tab:tab1} and \ref{tab:tab2} summarize the main characteristics of these classrooms. 

\begin{table}[!htp]
\centering
\caption{Classroom Characteristics}\label{tab:tab1}
\small 
\begin{adjustbox}{width=\textwidth} 
\begin{tabular}{c|c|c|c|c|c|c}
\hline 
\textbf{Classroom} & \textbf{$|S|$} & \textbf{\textbf{$\sum_{i<j} c_{ij}$}
} & \textbf{$\sum_{i}\max\{r_i,0\}$} & \textbf{$\sum_{i}|\min\{r_i,0\}|$} & $|\Lambda|$ & \textbf{$n_{\lambda}$}\\ \hline
I &33  &32  &9  &2  &7 & [4, 4, 5, 6, 4, 6, 4] \\ \hline
II &32  &88  &8  &8  &7 & [4, 4, 5, 5, 5, 5, 4] \\ \hline
III &31  &53  &4  &2  &6 & [5, 5, 5, 5, 5, 6] \\ \hline
\end{tabular}
\end{adjustbox}
\end{table}
The columns of Table~\ref{tab:tab1} indicate, respectively, the label of the classroom of instances, the number of students, the number of conflicts, the number of students to be in the front and back desks, the number of rows (layer) in the classroom, and the number of students per row. Table~\ref{tab:tab2} shows the IDs of the students that require sitting at the front or back of the classroom for each classroom.

\begin{table}[!htp]
\centering
\caption{Classroom Data Summary-Assigned Seating Preferences for Students}
\small 
\begin{adjustbox}{width=\textwidth} 
\begin{tabular}{c|c|c|c}
\hline
\textbf{Data Type} & \textbf{Classroom I} & \textbf{Classroom II} & \textbf{Classroom III} \\ \hline
\textbf{\shortstack{Students \\ in Front}} & 5, 6, 8, 10, 15, 16, 19, 20, 29 & 1, 4, 6, 18, 19, 23, 25, 31 & 2, 4, 7, 21 \\ \hline
\textbf{\shortstack{Students \\ in Back}} & 21, 23 & 5, 10, 13, 22, 26, 27, 28, 29 & 3, 27 \\ \hline
\end{tabular}
\end{adjustbox}
\label{tab:tab2}
\end{table}

Table~\ref{tab:tab2} presents a summary of students’ seating preferences based on their required positions within the classroom. Each column corresponds to one of the three classrooms analyzed (Classroom I, II, and III), while the rows indicate the students who requested to be seated either at the front or at the back. The numbers listed represent the student IDs associated with these specific preferences. This information is essential for the seat allocation process, as it ensures that individual seating constraints are respected according to the classroom layout.

Figure \ref{fig:fig3.1} displays a visualization of student conflicts in a graph representation, where the vertices indicate students and edges exist if there are conflicts between them. A complete list of conflicts is available in Table~1 of the Supplementary Material.

\begin{figure}[!htb]
    \centering
    \begin{subfigure}[b]{0.44\textwidth} 
        \centering
        \includegraphics[width=\textwidth]{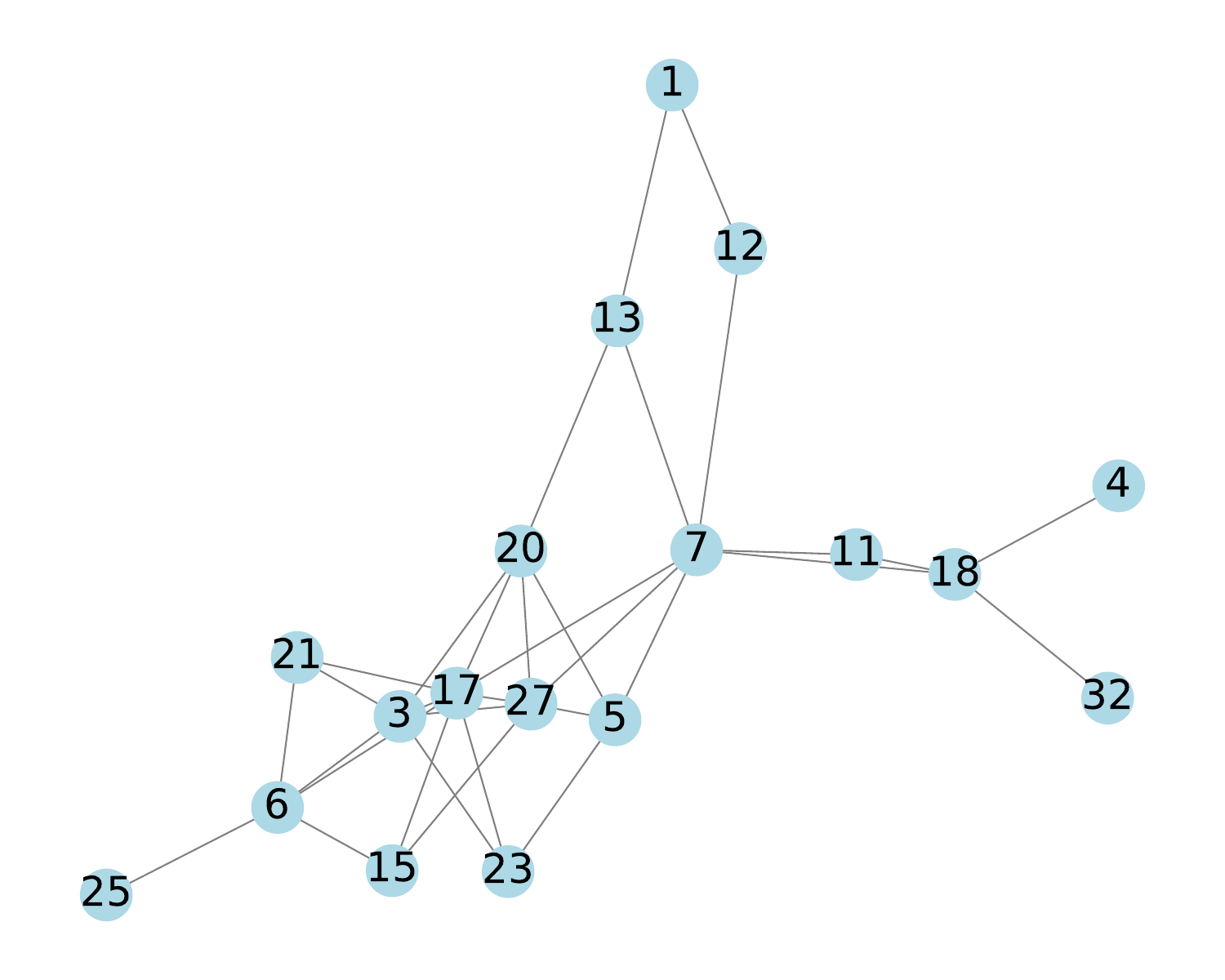}
        \caption{Graph of conflicts -- Classroom I}
        \label{fig:c1}
    \end{subfigure}
    \hspace{0.02\textwidth}
    \begin{subfigure}[b]{0.44\textwidth} 
        \centering
        \includegraphics[width=\textwidth]{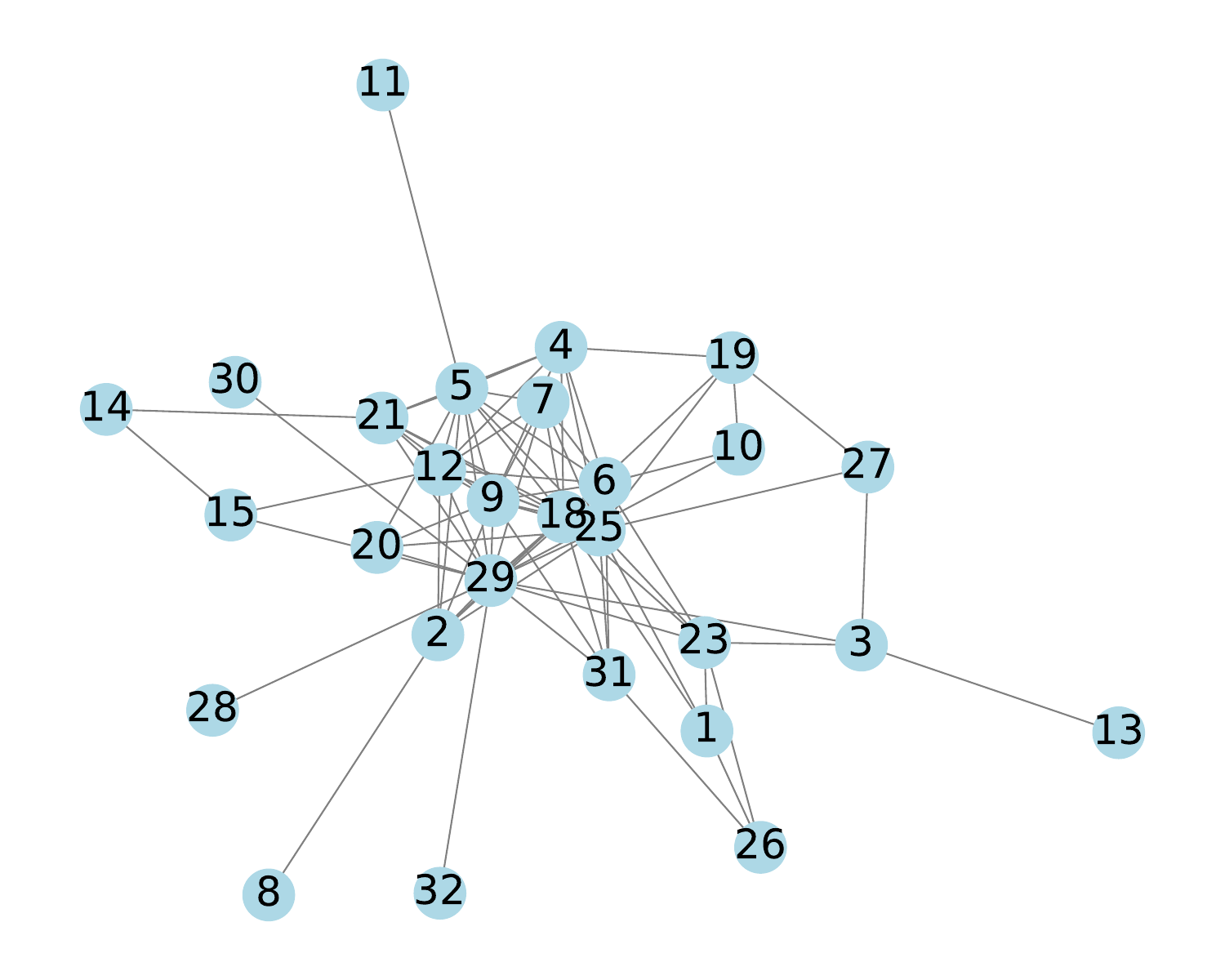}
        \caption{Graph of conflicts -- Classroom II}
        \label{fig:c2}
    \end{subfigure}
    \begin{subfigure}[b]{0.44\textwidth}
    \centering
    \includegraphics[width=\textwidth]{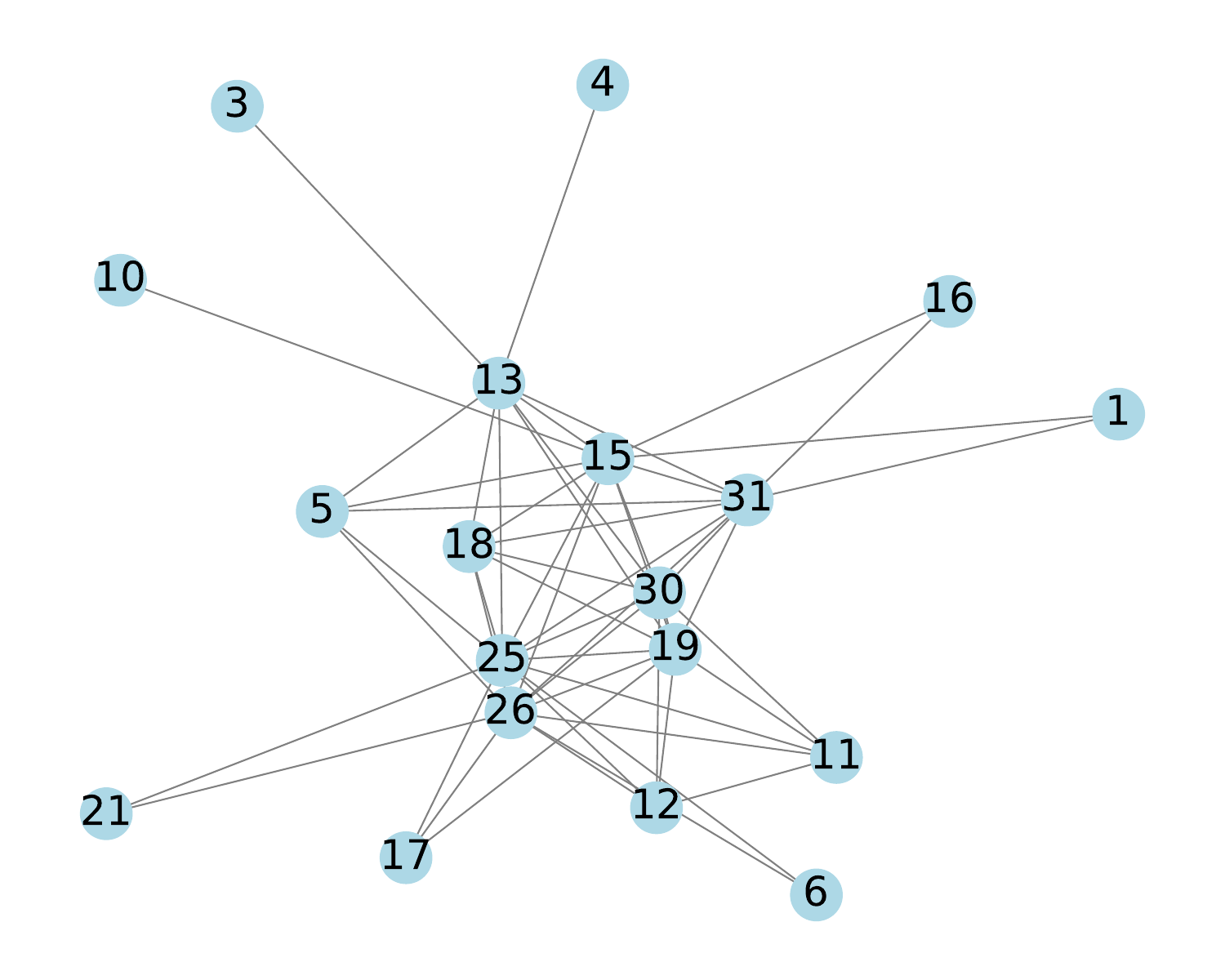}
\caption{Graph of conflicts -- Classroom III}\label{fig:c3}
   \end{subfigure}
\caption{Graphs representing the conflicts between students in three real data.}\label{fig:fig3.1}
 \end{figure}

The number of students involved in conflicts in Classroom I is 18 out of 33 students; in Classroom II is 28 out of 32 students; and in Classroom III is 19 out of 31 students.


To generate more instances to thoroughly evaluate the proposed heuristic method, we generate artificial instances. Random graph generation produces the graph of conflicts of the set of artificial instances. The next section presents a description of the data construction process.

\subsection{Artificial Instances}
\label{subsubsec:4.1.3}

The set of artificial instances was constructed using information from real classrooms in Brazilian public schools.   Table~\ref{tab:artificial_parameters} summarizes the parameters employed to generate the artificial instances, which reports the graph parameters, the layout of the class and student seating preferences.

\begin{table}[h]
\centering
{\footnotesize
\caption{Summary of the parameters used to generate artificial instances.}\label{tab:artificial_parameters}
\begin{tabular}{c|c|P{3.5cm}|P{6.0cm}}
\hline
\textbf{Cat.} & \textbf{Parameter} & \textbf{Employed Values} & \textbf{Justification / Outcome} \\
\hline
\multirow{4}{*}{Model/Instance} 
    & $n$ (students / desks) & 30, 35, 40 & Based on public school class sizes \\
\cline{2-4}
    & \% students in conflict & 35\%, 55\%, 85\%  & Controls number of students with conflicts \\
\cline{2-4}
    & \% conflicting edges & 30\%, 40\%, 50\%  & Adjusts conflict density among students \\
    \cline{2-4}
    & $d_{min}$,$d_{min}'$   & 2 &  Higher values reduce feasible placements and are limited by the shortest row\\
        \cline{2-4}
  & $\Psi$ &  $\max_{\lambda}n_{\lambda}$& A sufficiently large value, since $\Psi > \max_{\lambda}(n_{\lambda}-1)$\\
\hline
\multirow{2}{*}{Layout} 
    & Rows per class & Random: 5, 6, 7 & Reflects typical classroom layouts found in real educational settings
 \\
\cline{2-4}
    & Min. desks per row &  4 &  Avoids overlap in front/back classification \\
\hline
\multirow{2}{*}{Prefs.} 
    & Front preference rate & Random: 13\%–27\%  & Reflects realistic distribution of preferences \\
\cline{2-4}
    & Back preference rate & Random: 6\%–25\%  & Reflects realistic distribution of preferences  \\
\hline
\end{tabular}}
\end{table}

The methodology to construct the conflict graph for each instance consisted in using the information provided from real data and summarized in Table~\ref{tab:artificial_parameters} to create random graphs. For this, we used the $G_{\mathcal{NM}}$ model, considering the implementation available in the  NetworkX package. In this model, a graph is randomly selected from the set of all graphs with $\mathcal{N}$ vertices and $\mathcal{M}$ edges. To define $n$, we considered the values 30, 35, and 40, reflecting the typical maximum number of students per classroom in Brazilian public elementary schools. Therefore, the number of vertices of the conflict graph is proportional to the number of students in conflict (either 35\%, 55\%, or 85\%). The density of the conflict graph is defined by the percentage of conflicting relationships between students, set at 30\%, 40\%, or 50\%. Moreover, we impose that all vertices must have degree greater than or equal to one. 
Five graphs were generated for each of the 27 possible configurations (resulting from the combination of the number of vertices, the percentage of students in conflict, and the percentage of conflicting edges). 

After creating the conflicting graphs, the remaining parameters of the instances were defined considering the strategies and values (random or fixed values)  described in Table~\ref{tab:artificial_parameters}. 
To generate feasible classroom configurations for testing the proposed model, each row 
was required to contain at least four desks, ensuring a clear and non-overlapping distinction between front and back seats. The remaining $n - 4\cdot |\Lambda|$ seats were randomly distributed among the rows.

After excluding infeasible instances, a final total of 131 instances remained.
Tables 2 to 7  of the  Supplementary Material
provide information about the artificial instances. The data can be accessed and downloaded from the GitHub page \url{https://github.com/bcbraga/SSAP}.

\subsection{Computational and Parameter Setting}
\label{subsubsec:4.2}

Experiments with artificial instances were run in a high-performance cluster with four dedicated fat nodes, each of them equipped with two Intel Xeon E5-2667v4 processors and 512 GB of DDR3 memory.  
The second experiment, on real data,  was carried out in a personal computer to evaluate the performance of the heuristic on computers for individual use. The device was equipped with an Intel Core i5-10500H processor, featuring 6 cores and 12 threads, with a base clock speed of 2.50 GHz and turbo boost up to 3.87 GHz. 

The heuristic was implemented in Python 3.10 using PyCharm 2022.1.2 (Community Edition). Table~\ref{parametervalues} presents the  values tested and employed to the following parameters of ILS, Swap I and SSAP: perturbation intensity ($\theta$) for which random and fixed values were evaluated, being the tested fixed value the one suggested by \citep{Maximo2021}; the total number of iterations ($it_{max}$) and number of iterations without improvement ($\eta_{max}$);  the minimum distance between a pair of conflicting students  within the same row parameter (\( d_{min}\));  the penalty factor in the objective function (\( \phi \)); the  percentage of students selected per row for potential swaps in Swap I ($\psi$); and percentage of  seats evaluated in Swap I for students without any specific seating preference ($\Gamma$). 

\begin{table}[h]
\centering
{\footnotesize
\caption{Summary of parameters, values, and justifications}\label{parametervalues}
\begin{tabular}{c|c|P{4.2cm}|c|P{5.8cm}}
\hline
\textbf{Env.} & \textbf{Param.} & \textbf{Values/ Strategies Tested} &  \textbf{Employed}&\textbf{Justification / Outcome} \\
\hline
\multirow{3}{*}{ILS} 
    & $\theta$  & 
      Random: $\theta \sim \mathcal{U}[10\%, 50\%]$; Fixed: $\theta = 25\%$ & $25\%$&
      The fixed value provided more consistent and higher-quality results \\
\cline{2-5}
    & $it_{max}$ & 1,000; 3,000; 6,000; 10,000; 15,000; 21,000 &
      10,000 &       
      The employed value ensured a balance between solution quality and computational cost \\
\cline{2-5}
    & $\eta_{max}$ & 100; 500; 1,000; 1,500; 2,100; 2,800; 3,600; 4,500 & 
      500 & 
      This value controlled search stagnation and runtime \\
\hline
\multirow{2}{*}{Swap I} 
    & $\psi$ & 
      0.15; 0.25; 0.35; 0.45; 0.55; 0.65; 0.75 & 
      0.35 & This value gave best results when combined with candidate filtering \\
\cline{2-5}
    & $\Gamma$  & 
      0.15; 0.25; 0.35; 0.45; 0.55; 0.65; 0.75, with evaluation limited to $[8,\ 30]$ positions & 
      0.35 & This percentage guaranteed a satisfactory balance between solution quality and computational effort \\
\hline
\end{tabular}}
\label{tab:parameter_summary}
\end{table}

To set the perturbation degree $\theta$, we conducted an experimental analysis comparing a random selection drawn from the uniform distribution over $[10\%, 50\%]$ against a fixed value of $\theta = 25\%$. The latter strategy yielded better results, as suggested by \citep{Maximo2021}.

To ensure rigorous and systematic calibration of the remaining ILS parameters — $it_{max}$, $\eta_{max}$, $\psi$, and $\Gamma$ — we employed the automated configuration tool irace \citep{LOPEZIBANEZ201643}. The candidate configurations evaluated are reported in Table~\ref{parametervalues}, and the complete set of instances was used throughout this tuning phase.

The computational budget parameters $it_{max}$ and $\eta_{max}$ converged to 10{,}000 and 500, respectively. For the Swap I neighborhood, irace selected $\psi = 0.35$ as the threshold for random diversification target selection and $\Gamma = 0.35$ for candidate position filtering.

The penalty parameter $\phi$ is calibrated to guarantee that an infeasible solution is never evaluated more favorably than any feasible solution. For any feasible solution, the upper bound of the objective function is $0$ (occurring when no active edges exist in the assignment). A straightforward lower bound occurs when all edges are active and the distance between every pair of students is  $d_{min}$. In this scenario, the lower bound of the objective function is $|E|(d_{min} - \Psi)$. We parametrized  $\Psi$ to ensure $\Psi > d_{min}$; consequently, this lower bound is strictly negative. Since the minimum positive value for the total constraint violation $(\alpha+\beta+\gamma+\delta)$ is $1$, we define the penalty $\phi$ based on the worst-case (lowest) possible value of the objective function. Specifically, we set $\phi = 2 \cdot |E||(d_{min}-\Psi)|$. This conservative scaling ensures that any local search movement reducing constraint violations is strictly prioritized, thereby effectively guiding the refinement process toward a feasible solution space.

\subsection{Experiment I}
\label{subsec:sample_5_2}

This experiment reports the performance of ILS compared to Gurobi 10.0.3 time limited in 3600 seconds on the 131 artificial instances.
The performance metrics employed to compare the results obtained by the heuristic and the optimization solver are basically gap and execution time. The gap is computed with respect to the best-known solution (BKS) obtained either in a given run of ILS or Gurobi. Equation~\eqref{eq:gap_formula} presents how the gap is computed,  which is the difference between the  solution achieved by the evaluated method ($z_{primal}$) and the best-known solution ($z_{BKS}$).

\begin{equation}
gap = \frac{|z_{BKS} - z_{primal}|}{|z_{primal}| + 10^{-10} }
\label{eq:gap_formula}
\end{equation}

A smaller gap value indicates that the employed method is performing well and finding solutions very close to the best available reference. The addition of $10^{-10}$ is a practical numerical trick to make the formula robust against edge cases where the denominator might otherwise be zero.

For the heuristic, 30 independent executions were performed per instance. The reported gap and execution time correspond to the average over runs that produced feasible solutions. The feasibility rate (Fea) indicates the percentage of feasible runs out of the 30 executions for each instance.

Table~\ref{tab:table_DII_new} presents these performance metrics for ILS alongside the values obtained by Gurobi. In these tables, lighter gray cells highlight instances where the solver achieved better objective values than the heuristic's average over 30 runs. Darker gray cells, in contrast, indicate instances where the heuristic presented a better gap than the solver.

Let us first compare the performance of ILS and  the Gurobi solver. To facilitate the discussion, the 131 instances are divided according to their structural complexity into small- and medium-scale scenarios (IDs 1–30, 46–75, and 90–119) and large-scale, high-complexity scenarios (IDs 31–45, 76–89, and 120–131). This analysis focuses on three main aspects: the ability to find feasible solutions, the quality of the obtained solutions (measured by the gap), and the computational time.


Both approaches converged to the optimal solution in 78 small- and medium-scale instances, representing 60\% of the benchmark set. In these scenarios, Gurobi achieved optimality within seconds, requiring at most 10.50 seconds (Instance 118), while the full ILS framework matched this solution quality in a fraction of that time — running in a few hundredths of a second in the vast majority of cases, despite a maximum average execution time of 12.68 seconds observed for Instance 72.

For large-scale instances, Gurobi exhausted the time limit, returning feasible solutions without proving optimality, in 34 instances (26\% of the benchmark set). These cases are concentrated among the groups with the largest number of conflicting students: instances 32, 35–45, 76, 77, 79–89, and 121–131. Despite reaching the time limit, Gurobi closed the optimality gap in the majority of instances within IDs 31–45 and achieved very low gaps across the 76–89 subset, though it frequently required the full time budget to do so. Within the first complex subset (IDs 31–45), ILS produced its highest and most variable gaps, ranging from 0.30 (Instance 44) to 0.97 (Instances 33 and 39). In contrast, for the second subset (IDs 76–89), heuristic performance stabilized at lower gap values, remaining mostly below 0.43 (Instance 84), with several instances yielding even smaller gaps. The final subset (IDs 120–131) reveals a more balanced picture. Gurobi solved Instances 120 and 123 to proven optimality within the time limit, while for the remaining instances in this block, the gaps of both methods remained relatively low — generally below 0.54 — except for two gaps of 1.00 observed for Instances 120 and 123 under ILS. Overall, across the full benchmark set, Gurobi achieved a better gap in 45 instances (34\% of the instances), whereas ILS outperformed the solver in 8 instances (6\% of the instances): 83, 88, 89, 127, 128, 129, 130, and 131, highlighted in darker gray in Table~\ref{tab:table_DII_new}.

It should be noted, however, that the reported ILS gaps represent averages over 30 independent runs and may therefore be inflated by occasional unsuccessful executions. Examining the individual runs — detailed in the Supplementary Material — reveals that ILS frequently attained the optimal solution (gap equal to zero) even for large-scale instances in which Gurobi achieved a better average gap. Notably, ILS found the global optimum in 7 out of 30 runs for Instance 31, in 19 out of 30 runs for Instance 34, and in 20 out of 30 runs for Instance 78. This pattern is not isolated to these cases, as evidenced throughout Tables 8 and 9 of the Supplementary Material, and demonstrates that the heuristic possesses the robustness required to converge to the exact global optimum under dense conflict constraints, offering a competitive alternative at a fraction of the computational cost.

As an overall assessment, ILS obtained feasible solutions across all 30 runs for approximately 86\% of the instances. With respect to computational time, ILS was faster than Gurobi in the vast majority of cases; the commercial solver only outperformed the heuristic in Instances 34, 71, 72, 78, 109, 114, 116, and 119, where optimal solutions were reached very quickly.

\begin{landscape}
\renewcommand{\arraystretch}{1.3}
\begin{table}[!htp]
\centering
\caption{Performance measures by instance from the solutions obtained by ILS and Gurobi.}
\label{tab:table_DII_new}
\fontsize{10}{10}\selectfont
\resizebox{1.3\textwidth}{!}{
\begin{tabular}{c|c|c|c|c|c|c||c|c|c|c|c|c|c||c|c|c|c|c|c|c}
\hline
\textbf{ID} & \makecell{\textbf{ILS} \\ \textbf{gap}} & \makecell{\textbf{Gurobi} \\ \textbf{gap}} & \textbf{BKS} & \makecell{\textbf{ILS} \\ \textbf{Time}} & \makecell{\textbf{Gurobi} \\ \textbf{Time}} & \textbf{\(\%\) Fea} & \textbf{ID} & \makecell{\textbf{ILS} \\ \textbf{gap}} & \makecell{\textbf{Gurobi} \\ \textbf{gap}} & \textbf{BKS} & \makecell{\textbf{ILS} \\ \textbf{Time}} & \makecell{\textbf{Gurobi} \\ \textbf{Time}} & \textbf{\(\%\) Fea} & \textbf{ID} & \makecell{\textbf{ILS} \\ \textbf{gap}} & \makecell{\textbf{Gurobi} \\ \textbf{gap}} & \textbf{BKS} & \makecell{\textbf{ILS} \\ \textbf{Time}} & \makecell{\textbf{Gurobi} \\ \textbf{Time}} & \textbf{\(\%\) Fea} \\ \hline
1 & 0.00 & 0.00 & 0.00 & 0.02 & 0.51 & 100.00 & 45 & 0.19 & \cellcolor{lightgray} 0.00 & -29.00 & 20.95 & 3604.74 & 50.00 & 89 & \cellcolor{gray} 0.20 & 0.22 & -68.00 & 49.73 & 3609.32 & 63.33 \\
2 & 0.00 & 0.00 & 0.00 & 0.02 & 0.67 & 100.00 & 46 & 0.00 & 0.00 & 0.00 & 0.02 & 1.55 & 100.00 & 90 & 0.00 & 0.00 & 0.00 & 0.03 & 2.01 & 100.00 \\
3 & 0.00 & 0.00 & 0.00 & 0.01 & 0.89 & 100.00 & 47 & 0.00 & 0.00 & 0.00 & 0.03 & 1.36 & 100.00 & 91 & 0.00 & 0.00 & 0.00 & 0.07 & 2.63 & 100.00 \\
4 & 0.00 & 0.00 & 0.00 & 0.01 & 0.52 & 100.00 & 48 & 0.00 & 0.00 & 0.00 & 0.03 & 1.34 & 100.00 & 92 & 0.00 & 0.00 & 0.00 & 0.07 & 2.41 & 100.00 \\
5 & 0.00 & 0.00 & 0.00 & 0.01 & 0.59 & 100.00 & 49 & 0.00 & 0.00 & 0.00 & 0.02 & 1.17 & 100.00 & 93 & 0.00 & 0.00 & 0.00 & 0.06 & 2.52 & 100.00 \\
6 & 0.00 & 0.00 & 0.00 & 0.02 & 0.96 & 100.00 & 50 & 0.00 & 0.00 & 0.00 & 0.02 & 2.53 & 100.00 & 94 & 0.00 & 0.00 & 0.00 & 0.03 & 2.34 & 100.00 \\
7 & 0.00 & 0.00 & 0.00 & 0.02 & 0.80 & 100.00 & 51 & 0.00 & 0.00 & 0.00 & 0.02 & 2.46 & 100.00 & 95 & 0.00 & 0.00 & 0.00 & 0.04 & 2.99 & 100.00 \\
8 & 0.00 & 0.00 & 0.00 & 0.01 & 0.84 & 100.00 & 52 & 0.00 & 0.00 & 0.00 & 0.02 & 2.43 & 100.00 & 96 & 0.00 & 0.00 & 0.00 & 0.03 & 2.55 & 100.00 \\
9 & 0.00 & 0.00 & 0.00 & 0.01 & 0.67 & 100.00 & 53 & 0.00 & 0.00 & 0.00 & 0.02 & 2.19 & 100.00 & 97 & 0.00 & 0.00 & 0.00 & 0.05 & 2.16 & 100.00 \\
10 & 0.00 & 0.00 & 0.00 & 0.01 & 1.24 & 100.00 & 54 & 0.00 & 0.00 & 0.00 & 0.02 & 2.62 & 100.00 & 98 & 0.00 & 0.00 & 0.00 & 0.04 & 2.29 & 100.00 \\
11 & 0.00 & 0.00 & 0.00 & 0.02 & 0.74 & 100.00 & 55 & 0.00 & 0.00 & 0.00 & 0.02 & 2.41 & 100.00 & 99 & 0.00 & 0.00 & 0.00 & 0.03 & 2.74 & 100.00 \\
12 & 0.00 & 0.00 & 0.00 & 0.02 & 0.90 & 100.00 & 56 & 0.00 & 0.00 & 0.00 & 0.04 & 4.28 & 100.00 & 100 & 0.00 & 0.00 & 0.00 & 0.04 & 2.71 & 100.00 \\
13 & 0.00 & 0.00 & 0.00 & 0.02 & 1.00 & 100.00 & 57 & 0.00 & 0.00 & 0.00 & 0.04 & 3.85 & 100.00 & 101 & 0.00 & 0.00 & 0.00 & 0.07 & 3.30 & 100.00 \\
14 & 0.00 & 0.00 & 0.00 & 0.02 & 0.75 & 100.00 & 58 & 0.00 & 0.00 & 0.00 & 0.02 & 3.59 & 100.00 & 102 & 0.00 & 0.00 & 0.00 & 0.08 & 3.54 & 100.00 \\
15 & 0.00 & 0.00 & 0.00 & 0.01 & 0.93 & 100.00 & 59 & 0.00 & 0.00 & 0.00 & 0.05 & 2.99 & 100.00 & 103 & 0.00 & 0.00 & 0.00 & 0.09 & 3.54 & 100.00 \\
16 & 0.00 & 0.00 & 0.00 & 0.06 & 1.86 & 100.00 & 60 & 0.00 & 0.00 & 0.00 & 0.02 & 2.67 & 100.00 & 104 & 0.00 & 0.00 & 0.00 & 0.04 & 3.48 & 100.00 \\
17 & 0.00 & 0.00 & 0.00 & 0.02 & 1.29 & 100.00 & 61 & 0.00 & 0.00 & 0.00 & 0.11 & 4.86 & 100.00 & 105 & 0.00 & 0.00 & 0.00 & 0.97 & 6.00 & 100.00 \\
18 & 0.00 & 0.00 & 0.00 & 0.03 & 1.43 & 100.00 & 62 & 0.00 & 0.00 & 0.00 & 0.07 & 4.96 & 100.00 & 106 & 0.00 & 0.00 & 0.00 & 1.45 & 5.17 & 100.00 \\
19 & 0.00 & 0.00 & 0.00 & 0.03 & 1.92 & 100.00 & 63 & 0.00 & 0.00 & 0.00 & 0.51 & 6.20 & 100.00 & 107 & 0.00 & 0.00 & 0.00 & 0.08 & 4.89 & 100.00 \\
20 & 0.00 & 0.00 & 0.00 & 0.03 & 1.76 & 100.00 & 64 & 0.00 & 0.00 & 0.00 & 0.16 & 4.65 & 100.00 & 108 & 0.00 & 0.00 & 0.00 & 0.09 & 5.12 & 100.00 \\
21 & 0.00 & 0.00 & 0.00 & 0.04 & 1.88 & 100.00 & 65 & 0.00 & 0.00 & 0.00 & 0.05 & 4.71 & 100.00 & 109 & 0.40 & \cellcolor{lightgray} 0.00 & 0.00 & 13.60 & 7.29 & 100.00 \\
22 & 0.00 & 0.00 & 0.00 & 0.05 & 2.37 & 100.00 & 66 & 0.00 & 0.00 & 0.00 & 0.22 & 6.54 & 100.00 & 110 & 0.03 & \cellcolor{lightgray} 0.00 & 0.00 & 3.43 & 6.16 & 100.00 \\
23 & 0.00 & 0.00 & 0.00 & 0.05 & 1.85 & 100.00 & 67 & 0.07 & \cellcolor{lightgray} 0.00 & 0.00 & 1.73 & 8.49 & 100.00 & 111 & 0.00 & 0.00 & 0.00 & 0.65 & 8.95 & 100.00 \\
24 & 0.00 & 0.00 & 0.00 & 0.09 & 2.08 & 100.00 & 68 & 0.00 & 0.00 & 0.00 & 1.12 & 5.59 & 100.00 & 112 & 0.00 & 0.00 & 0.00 & 0.25 & 5.68 & 100.00 \\
25 & 0.00 & 0.00 & 0.00 & 0.12 & 2.03 & 100.00 & 69 & 0.00 & 0.00 & 0.00 & 0.32 & 3.67 & 100.00 & 113 & 0.13 & \cellcolor{lightgray} 0.00 & 0.00 & 8.28 & 10.11 & 100.00 \\
26 & 0.00 & 0.00 & 0.00 & 0.18 & 2.53 & 100.00 & 70 & 0.00 & 0.00 & 0.00 & 0.13 & 3.81 & 100.00 & 114 & 0.40 & \cellcolor{lightgray} 0.00 & 0.00 & 12.25 & 8.18 & 100.00 \\
27 & 0.03 & \cellcolor{lightgray} 0.00 & 0.00 & 0.53 & 3.03 & 100.00 & 71 & 0.20 & \cellcolor{lightgray} 0.00 & 0.00 & 7.17 & 4.73 & 100.00 & 115 & 0.00 & 0.00 & 0.00 & 3.93 & 8.42 & 100.00 \\
28 & 0.00 & 0.00 & 0.00 & 0.61 & 2.80 & 100.00 & 72 & 0.50 & \cellcolor{lightgray} 0.00 & 0.00 & 12.68 & 6.81 & 100.00 & 116 & 0.17 & \cellcolor{lightgray} 0.00 & 0.00 & 9.28 & 8.57 & 100.00 \\
29 & 0.00 & 0.00 & 0.00 & 1.00 & 2.22 & 100.00 & 73 & 0.03 & \cellcolor{lightgray} 0.00 & 0.00 & 1.20 & 4.52 & 100.00 & 117 & 0.03 & \cellcolor{lightgray} 0.00 & 0.00 & 5.45 & 8.59 & 100.00 \\
30 & 0.00 & 0.00 & 0.00 & 1.00 & 2.81 & 100.00 & 74 & 0.00 & 0.00 & 0.00 & 2.42 & 5.64 & 100.00 & 118 & 0.00 & 0.00 & 0.00 & 0.12 & 10.50 & 100.00 \\
31 & 0.77 & \cellcolor{lightgray} 0.00 & 0.00 & 12.76 & 13.11 & 100.00 & 75 & 0.00 & 0.00 & 0.00 & 0.22 & 4.06 & 100.00 & 119 & 0.40 & \cellcolor{lightgray} 0.00 & 0.00 & 17.26 & 8.85 & 100.00 \\
32 & 0.33 & \cellcolor{lightgray} 0.00 & -18.00 & 21.56 & 3604.02 & 100.00 & 76 & 0.13 & \cellcolor{lightgray} 0.00 & -88.00 & 38.93 & 3607.48 & 93.33 & 120 & 1.00 & \cellcolor{lightgray} 0.00 & 0.00 & 56.84 & 2105.74 & 100.00 \\
33 & 0.97 & \cellcolor{lightgray} 0.00 & 0.00 & 13.41 & 16.03 & 100.00 & 77 & 0.39 & \cellcolor{lightgray} 0.00 & -22.00 & 39.31 & 3607.31 & 100.00 & 121 & 0.54 & \cellcolor{lightgray} 0.37 & -12.00 & 58.74 & 3609.55 & 100.00 \\
34 & 0.37 & \cellcolor{lightgray} 0.00 & 0.00 & 7.43 & 4.72 & 100.00 & 78 & 0.33 & \cellcolor{lightgray} 0.00 & 0.00 & 16.82 & 8.94 & 100.00 & 122 & 0.39 & \cellcolor{lightgray} 0.00 & -41.00 & 77.00 & 3611.41 & 100.00 \\
35 & 0.53 & \cellcolor{lightgray} 0.00 & -8.00 & 20.75 & 3603.88 & 100.00 & 79 & 0.38 & \cellcolor{lightgray} 0.18 & -14.00 & 43.16 & 3606.02 & 100.00 & 123 & 1.00 & \cellcolor{lightgray} 0.00 & 0.00 & 63.94 & 3039.65 & 100.00 \\
36 & 0.86 & \cellcolor{lightgray} 0.00 & -1.00 & 17.13 & 3604.50 & 100.00 & 80 & 0.34 & \cellcolor{lightgray} 0.06 & -17.00 & 37.90 & 3606.71 & 100.00 & 124 & 0.25 & \cellcolor{lightgray} 0.16 & -48.00 & 72.94 & 3611.62 & 96.67 \\
37 & 0.54 & \cellcolor{lightgray} 0.00 & -7.00 & 18.95 & 3603.73 & 100.00 & 81 & 0.20 & \cellcolor{lightgray} 0.16 & -114.00 & 54.72 & 3610.47 & 73.33 & 125 & 0.31 & \cellcolor{lightgray} 0.17 & -59.00 & 89.31 & 3613.53 & 100.00 \\
38 & 0.45 & \cellcolor{lightgray} 0.00 & -20.00 & 29.16 & 3604.73 & 100.00 & 82 & 0.28 & \cellcolor{lightgray} 0.23 & -66.00 & 50.04 & 3609.48 & 100.00 & 126 & 0.19 & \cellcolor{lightgray} 0.15 & -105.00 & 101.83 & 3615.04 & 100.00 \\
39 & 0.97 & \cellcolor{lightgray} 0.00 & 0.00 & 16.44 & 33.09 & 100.00 & 83 & \cellcolor{gray} 0.29 & 0.32 & -39.00 & 42.95 & 3608.31 & 100.00 & 127 & \cellcolor{gray} 0.25 & 0.37 & -61.00 & 88.95 & 3613.23 & 96.67 \\
40 & 0.42 & \cellcolor{lightgray} 0.00 & -14.00 & 22.30 & 3604.39 & 100.00 & 84 & 0.43 & \cellcolor{lightgray} 0.00 & -13.00 & 34.50 & 3606.87 & 100.00 & 128 & \cellcolor{gray} 0.17 & 0.22 & -175.00 & 108.50 & 3616.59 & 93.33 \\
41 & 0.32 & \cellcolor{lightgray} 0.00 & -22.00 & 24.01 & 3604.86 & 60.00 & 85 & 0.39 & \cellcolor{lightgray} 0.24 & -37.00 & 44.79 & 3608.34 & 93.33 & 129 & \cellcolor{gray} 0.00 & 0.16 & -307.00 & 95.94 & 3620.01 & 3.33 \\
42 & 0.36 & \cellcolor{lightgray} 0.06 & -16.00 & 24.46 & 3604.53 & 73.33 & 86 & 0.30 & \cellcolor{lightgray} 0.24 & -39.00 & 48.08 & 3608.42 & 60.00 & 130 & \cellcolor{gray} 0.25 & 0.41 & -88.00 & 91.59 & 3614.66 & 90.00 \\
43 & 0.32 & \cellcolor{lightgray} 0.00 & -50.00 & 25.04 & 3605.41 & 73.33 & 87 & 0.20 & \cellcolor{lightgray} 0.12 & -85.00 & 48.81 & 3610.04 & 56.67 & 131 & \cellcolor{gray} 0.14 & 0.30 & -146.00 & 129.32 & 3616.77 & 30.00 \\
44 & 0.30 & \cellcolor{lightgray} 0.09 & -32.00 & 27.78 & 3605.43 & 50.00 & 88 & \cellcolor{gray} 0.28 & 0.33 & -60.00 & 51.36 & 3608.70 & 96.67 & - & - & - & - & - & - & - \\ \hline
\end{tabular}
}
\end{table}
\end{landscape}



\subsection{Experiment II: Case Studies}
\label{subsec:sample_5_3}

This section describes the experiment applying the ILS heuristic to three real-world instances: Classrooms I, II, and III. In this experiment, the number of independent runs was ten, sufficient to evaluate the robustness of the heuristic. 
All runs yielded identical optimal objective values (zero) and similar runtimes in Classrooms I and III. In Classroom II, the final four runs achieved the same objective value and execution time, confirming the method's robustness.
Table~\ref{tab:res_runs} reports a single representative result for each instance.

\begin{table}[!htp]
\centering
\small
\caption{Comparison of objective values and execution times between ILS and Gurobi.}
\label{tab:res_runs}
\label{tab:tab4.1}

\begin{tabular}{c|c|c|c|c}
\hline
\textbf{Classroom} & \textbf{\shortstack{OF Gurobi}} & \textbf{\shortstack{Time(s) Gurobi}} & \textbf{\shortstack{OF ILS}} & \textbf{\shortstack{Time(s) ILS}} \\ \hline
I &0  &0.95  &0  &0.02   \\ 
II &0  &2.84  &0  &0.67 \\ 
III &0  & 1.91 &0  &0.05 \\ \hline
\end{tabular}

\end{table}


Both methods reached optimal solutions for the three instances. ILS required significantly lower execution times, particularly for Classrooms I and III. Although it took slightly longer for Classroom II, ILS was still nearly four times faster than Gurobi.

Finally, we compare the seating mappings generated by  the ILS heuristic with those manually designed by a team of teachers. Figure~\ref{fig:ils_teachers_classrooms} presents these side-by-side comparisons for the three real-world classroom scenarios.



\begin{figure}[htp]
    \centering

    \begin{subfigure}[b]{0.48\textwidth}
        \centering
        \includegraphics[width=\textwidth]{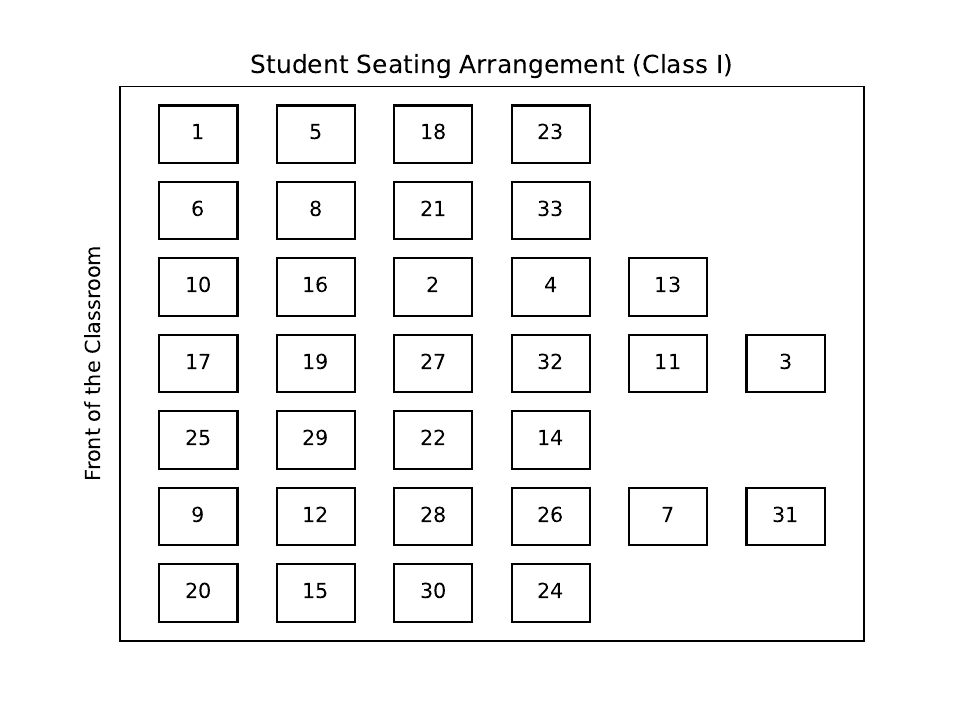}
        \caption{ILS - Classroom I}
        \label{fig:4.1c2}
    \end{subfigure}
    \hfill
    \begin{subfigure}[b]{0.48\textwidth}
        \centering
        \includegraphics[width=\textwidth]{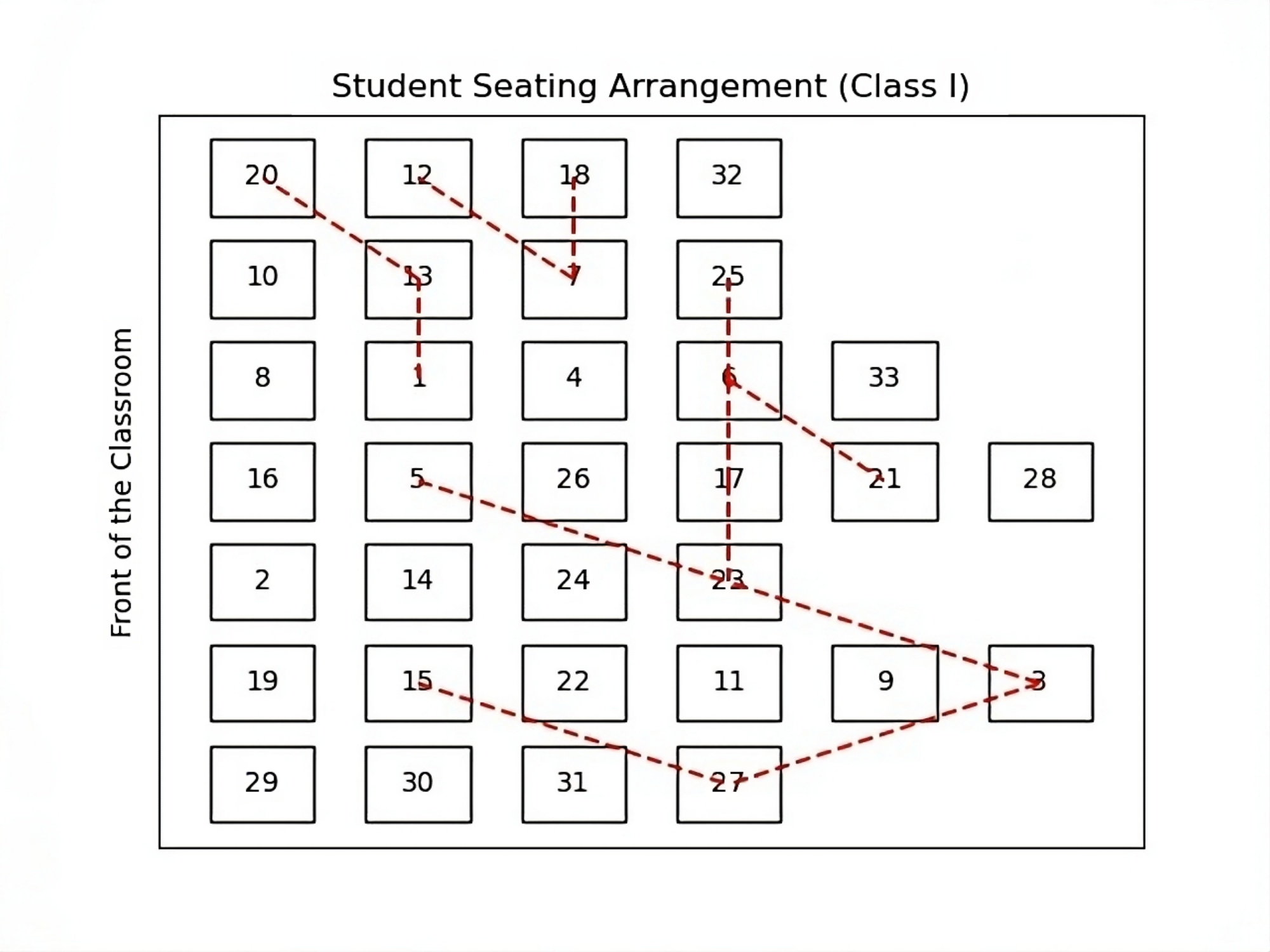}
        \caption{Teachers - Classroom I}
        \label{fig:4.1c3}
    \end{subfigure}

    \vspace{1em}

    \begin{subfigure}[b]{0.48\textwidth}
        \centering
        \includegraphics[width=\textwidth]{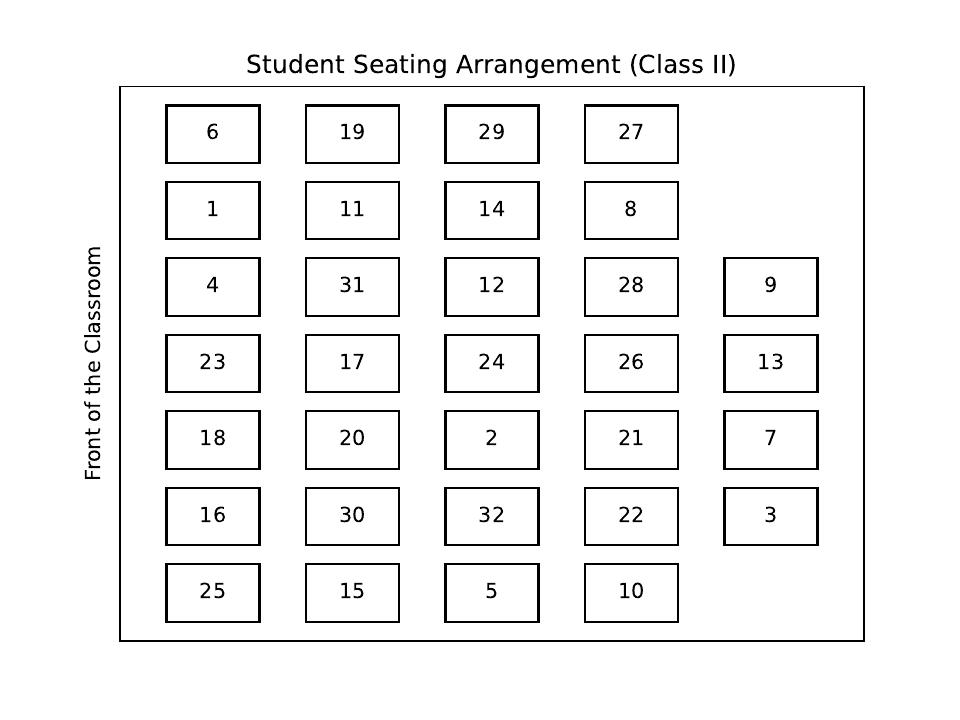}
        \caption{ILS - Classroom II}
        \label{fig:4.2c2}
    \end{subfigure}
    \hfill
    \begin{subfigure}[b]{0.48\textwidth}
        \centering
        \includegraphics[width=\textwidth]{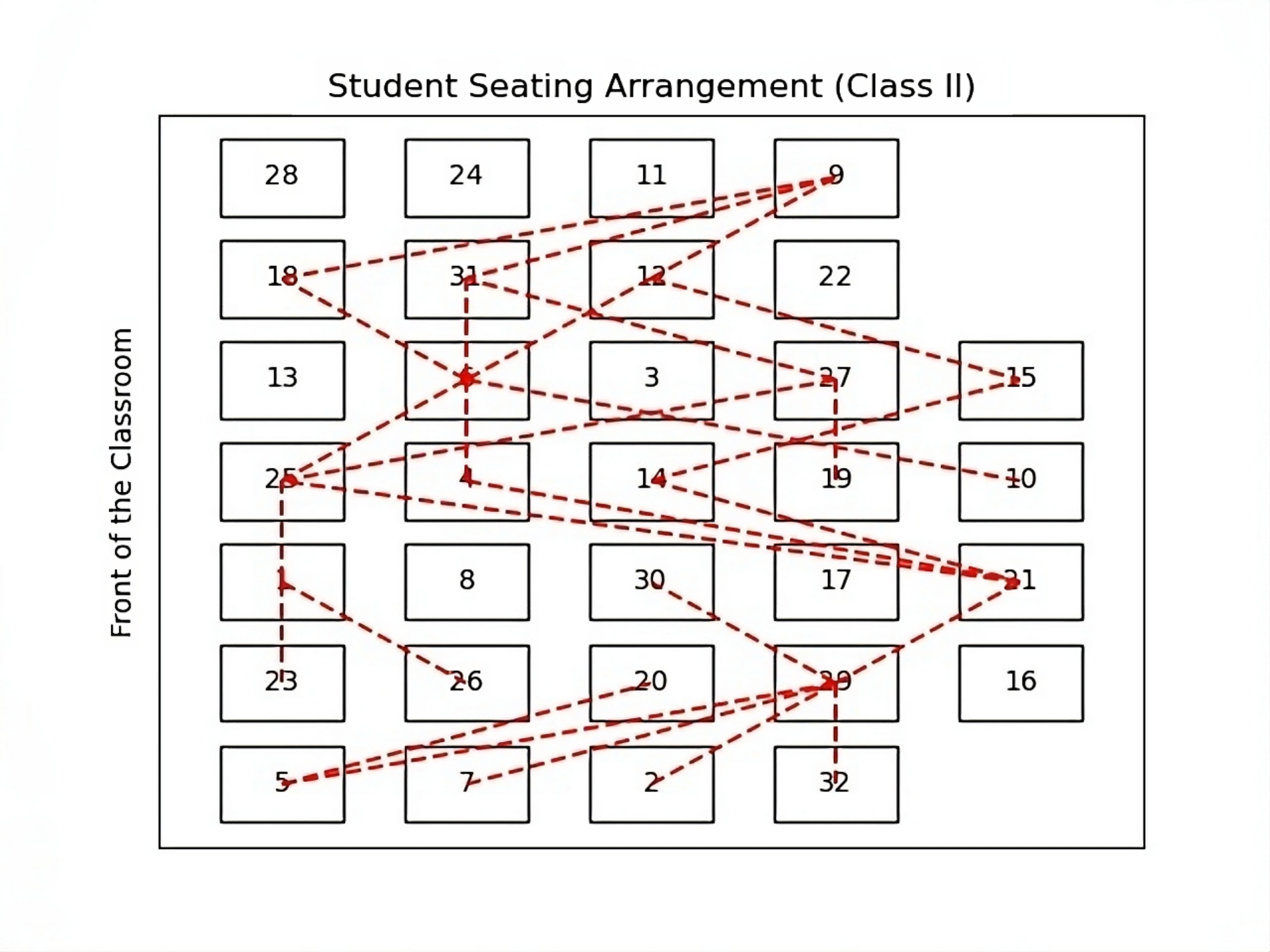}
        \caption{Teachers - Classroom II}
        \label{fig:4.2c3}
    \end{subfigure}

    \vspace{1em}

    \begin{subfigure}[b]{0.48\textwidth}
        \centering
        \includegraphics[width=\textwidth]{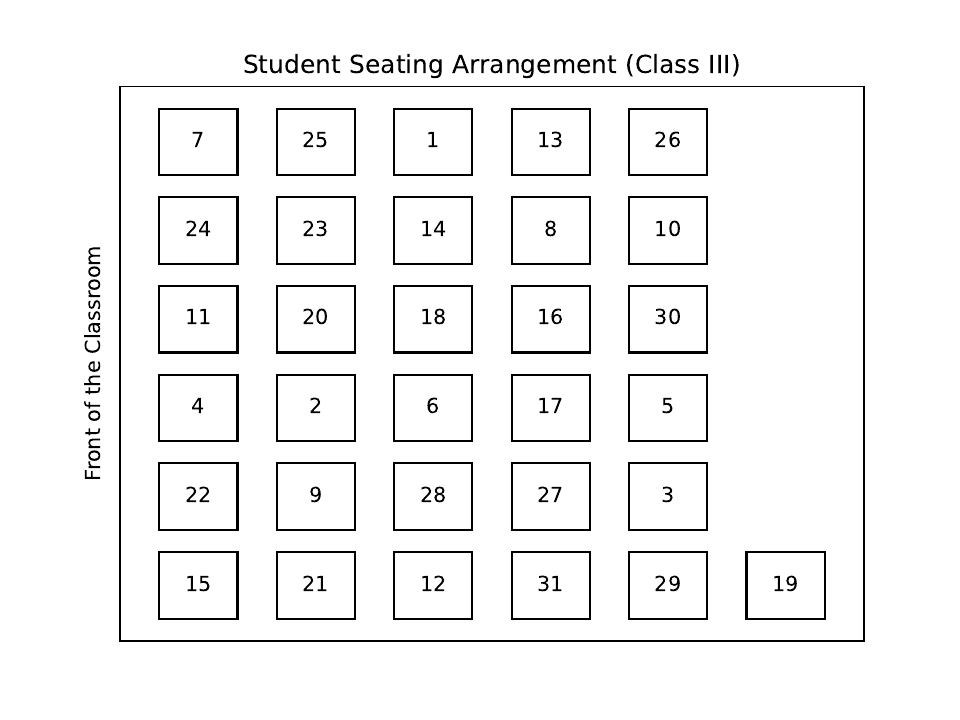}
        \caption{ILS - Classroom III}
        \label{fig:4.3c2}
    \end{subfigure}
    \hfill
    \begin{subfigure}[b]{0.48\textwidth}
        \centering
        \includegraphics[width=\textwidth]{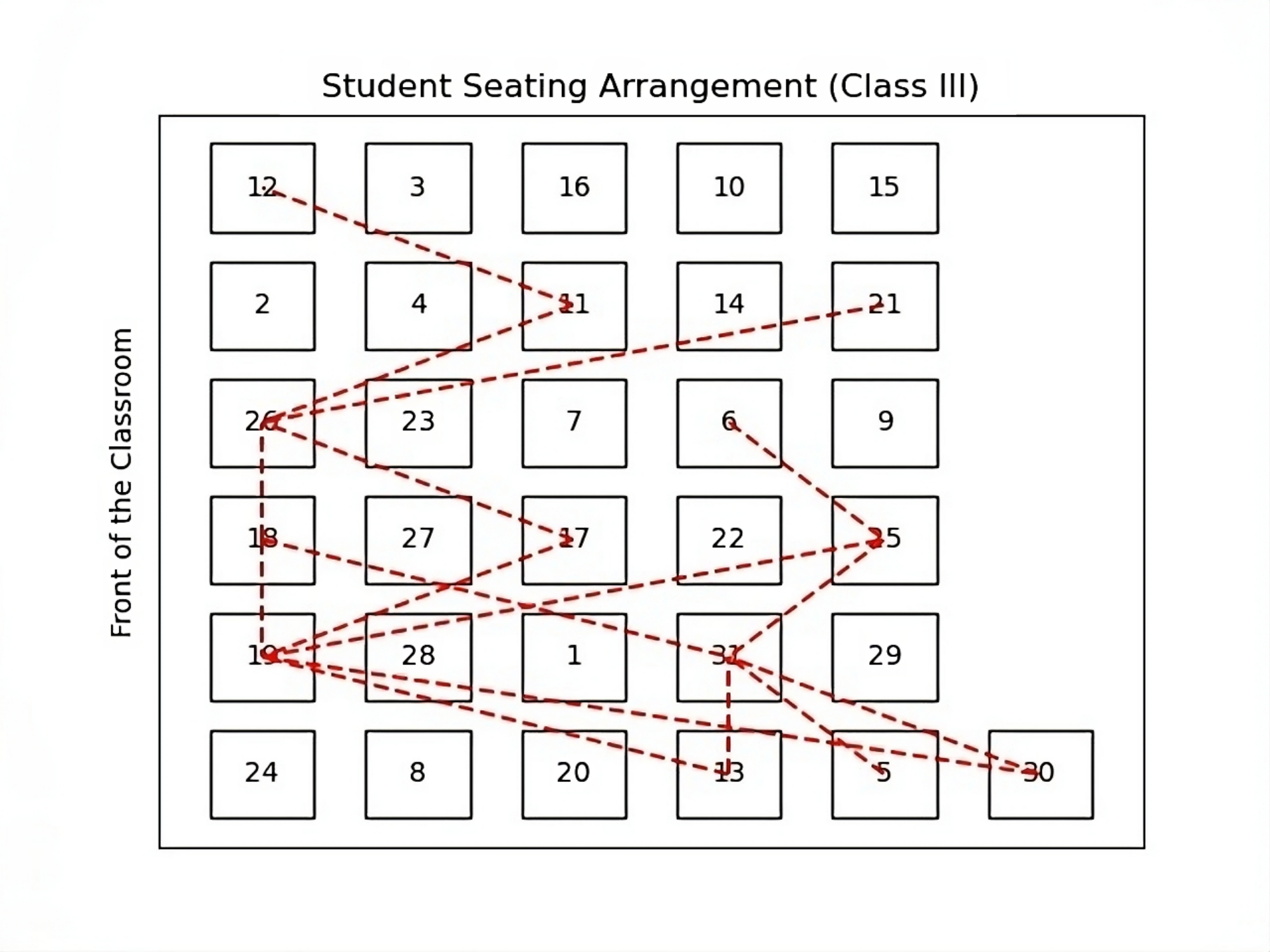}
        \caption{Teachers - Classroom III}
        \label{fig:4.3c3}
    \end{subfigure}

    \caption{Representation of student seat mapping for Classrooms I, II, and III: (a, c, and e) generated by ILS, and (b, d, and f) proposed by a group of teachers. Red lines indicate active edges in the conflict graph.}
    \label{fig:ils_teachers_classrooms}
\end{figure}


According to Figure~\ref{fig:ils_teachers_classrooms}, the ILS heuristic successfully generated classroom mappings that consistently avoided conflicts between students. This visual evidence not only reinforces the ILS's greater efficiency in achieving near-optimal solutions, as previously detailed in Table~\ref{tab:res_runs}, but also starkly contrasts with manual layouts. While teachers strive to minimize conflicts using diverse strategies, the complexity of the task often limits their ability to achieve an ideal solution.  ILS consistently produced solutions with zero active edges between conflicting students, effectively separating them, whereas manual mappings, despite reflecting teacher experience,  failed to meet all the conflicting problem's constraints.

For instance, Figures~\ref{fig:4.1c2} and \ref{fig:4.1c3} illustrate results for Classroom I, which has the fewest conflicts. The manually created assignment student-desk (Figure~\ref{fig:4.1c3}), though relatively organized, still showed 12 active edges, with only 4 satisfying the constraints. In contrast, Figure~\ref{fig:4.1c2} demonstrates the ILS's ability to achieve an assignment will without active edges. This pattern continues for more challenging cases: for Classroom II (Figures~\ref{fig:4.2c2} and \ref{fig:4.2c3}), which has the highest number of conflicts (88), the manual mapping displayed numerous active edges, underscoring the difficulty that teachers face in creating effective mappings. Conversely,  ILS obtained a configuration with zero active edges. Similarly, the assignment provided by the teachers for Classroom III (Figures~\ref{fig:4.3c2} and \ref{fig:4.3c3}), with an intermediate number of conflicts (53), violate minimum distance constraints and fail to achieve zero active edges, unlike the ILS solution.

Ultimately, this analysis demonstrates that the proposed heuristic can assist teachers in creating more efficient seating arrangements. By enabling the generation of multiple optimized solutions and facilitating targeted adjustments, this approach, through its consideration of model constraints and application of optimization techniques, offers a more objective and effective alternative for classroom organization.


\section{Conclusions and Future Work}
\label{sec:conclusions}

This paper approached the Student Seat Allocation Problem (SSAP), which involves  optimizing the student seating in a classroom with traditional arrangement. The goal is to assign students to seats in a way that minimizes conflict  by placing students with potential issues as far as apart as possible. The mathematical model introduced in this paper included constraints ensuring a minimum distance between seats assigned to conflicting students. The objective function involved minimizing active edges, meaning conflicting students seated in adjacent or in the same row, and maximizing the sum of their length. Besides, this formulation considers seat preferences where students are required to seat at front or back seats.

To heuristically solve the SSAP, we propose an ILS metaheuristic that returns a feasible solution whenever one is found; otherwise, it outputs the best relaxed solution obtained by penalizing constraint violations in the objective.

To validate the model and heuristic approach, the ILS heuristic and the Gurobi solver results were evaluated using a benchmark of 131 artificial instances. Both methods performed efficiently on small and medium-sized instances, successfully converging to the exact same optimal solution in 78 cases (60\% of the benchmark set). However, for large-scale instances, the exact solver frequently exhausted its 3,600-second time limit. Despite this, it achieved exactly zero gaps or very low gaps in most large-scale cases, outperforming the heuristic's average gap in 45 instances (34\% of the set). In these challenging scenarios, ILS average gaps fluctuated between 0.30 and 0.97 for instances 31–45, stabilized below 0.43 for instances 76–89, and outperformed Gurobi's bounds in 8 instances of the 120–131 block. Furthermore, the analysis of individual runs revealed that the heuristic successfully reached the global optimal solution (gap equal to zero) across multiple executions for instances 31, 34, and 78, proving to be a fast alternative that delivers solution quality within a fraction of the exact solver's computational time.

Beyond performance, accessibility is also a crucial factor. Gurobi is a commercial solver, which can limit its use in public educational institutions. In contrast, the implemented heuristic will be made publicly available and freely accessible for educational use by instructors. 

As future research, additional features could be incorporated to enhance the realism of the problem. One promising direction is the inclusion of constraints that group students by learning ability to foster peer collaboration. Another natural extension is to account for varying levels of interpersonal conflict: since conflicts differ in severity, introducing a graded conflict scale would allow the model to impose differentiated seating requirements — enforcing strict isolation for severe conflicts while applying more flexible distance thresholds for minor one. Moreover, future research could extend the SSAP to accommodate non-traditional classroom topologies, such as modular groupings or double-occupancy seatings.

\section*{Acknowledgments}
The authors thank the Coordenação de Aperfeiçoamento de Pessoal de Nível Superior/Brasil (CAPES) – Finance Code 001 – for the financial support.
We are also grateful for FAPESP (grant 2022/05803-3) and Conselho Nacional de Desenvolvimento Científico e Tecnológico (CNPq) (grants 305157/2025-6,403735/2021-1) for the financial support.
Research carried out using the computational resources of the Center for Mathematical Sciences Applied to Industry (CeMEAI) funded by FAPESP (grant 2013/07375-0).









\appendix

\section{Complexity of SSAP}\label{complexity}

\newtheorem{theorem}{Theorem}

\begin{theorem}
The Student Seat Allocation Problem (SSAP) is NP-hard.
\end{theorem}

\begin{proof}
We demonstrate the NP-hardness of SSAP through a polynomial-time reduction from the 2-Layer Linear Arrangement Problem (2LLA). The 2LLA problem involves a bipartite graph $G = (V_1 \cup V_2, E)$, where the vertices of $V_1$ and $V_2$ must be placed on two parallel lines (layers). The goal is to find a permutation (ordering) of the vertices in each layer that minimizes or maximizes a function of the edge lengths, such as the total horizontal displacement. This problem is known to be NP-hard for general bipartite graphs \citep{shahrokhi2001bipartite}. 

Consider a restriction of the SSAP to only two rows ($|\Lambda|=2$). In this specific case, the objective function (\ref{obj1}) simplifies to finding a permutation of students in two consecutive layers that maximizes the horizontal displacement $|z - k|$ for all conflicting pairs $(i,j) \in E$.

For any bipartite graph $G = (V_1 \cup V_2, E)$, we can map it to the SSAP by assigning a minimum distance value $d_{min}' > n_{\lambda}$ to all pairs of students within the same partition ($V_1$ or $V_2$). Moreover,  $d_{min}$ must be $1$.

Due to the constraint $d_{min}' > n_{\lambda}$, no two students from the same partition can be assigned to the same row, effectively forcing $V_1$ into row $\lambda$ and $V_2$ into row $\lambda+1$. The problem then reduces to finding an ordering (permutation) of nodes in two layers to optimize edge lengths.

Since the linear arrangement of bipartite graphs to optimize edge lengths is known to be NP-hard,  SSAP is also NP-hard.
\end{proof}

\section{Effectiveness of the Initial Solution}
\label{subsec:sample_ini}

To assess the impact and efficiency of the construction phase, we compare the quality of its solutions and computational time against those of the full ILS. Figures~\ref{fig:nuvem_gap} and~\ref{fig:nuvem_time} display the gap to the BKS and the average execution times, respectively, across all 131 instances. We distinguish among three solutions: the constructive phase yielding a feasible solution, labeled ``Initial''; the constructive phase yielding an infeasible solution, labeled ``Initial (Infeasible)''; and the solution obtained immediately after the intensification and diversification steps of ILS, labeled ``ILS with Seed''.

As shown in Figure~\ref{fig:nuvem_gap}, the proposed constructive heuristic provides high-quality starting points. In 13 instances (10\% of the total), the initial solution already matched the BKS. For the 36 instances in which the constructive method yields an infeasible initial solution, ILS was still able to find feasible solutions, achieving an average gap to the BKS of 33\%, with a best gap of 0 (in four instances) and a worst gap of 1 (in one instance).

\begin{figure}[!ht]
    \centering
    \includegraphics[width=\textwidth]{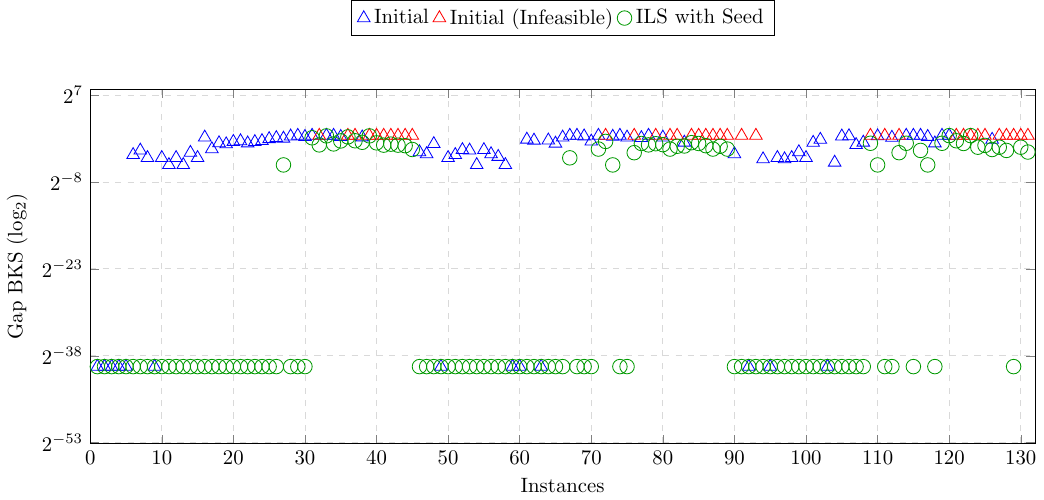}
    \caption{Average gap of the initial solution vs. average gap of the ILS solution.}
    \label{fig:nuvem_gap}
\end{figure}

\begin{figure}[!ht]
    \centering
    \includegraphics[width=\textwidth]{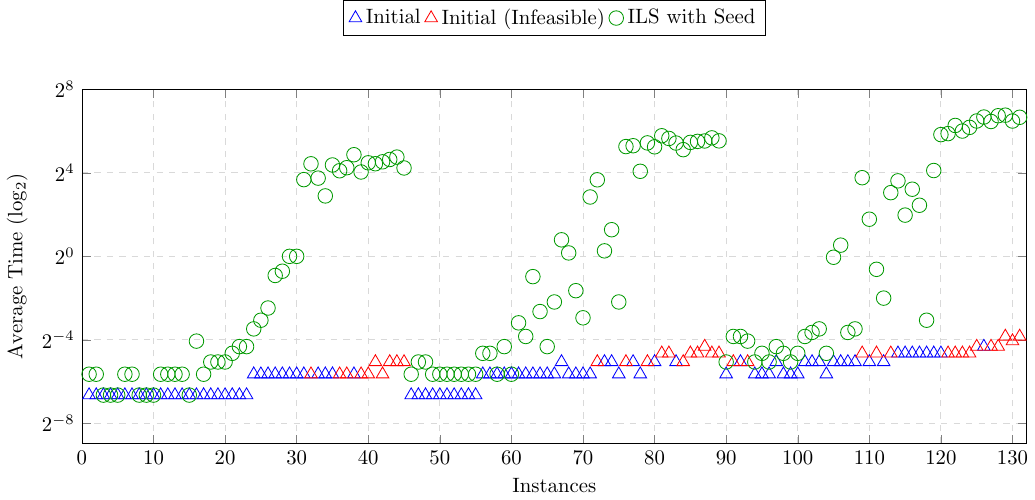}
    \caption{Average computational time of the initial constructive phase vs. average time of ILS with a Seed.}
    \label{fig:nuvem_time}
\end{figure}

Regarding computational time, Figure~\ref{fig:nuvem_time} shows that the effort devoted to the initialization phase is practically negligible. The constructive heuristic is exceptionally fast, requiring on average only 0.02 seconds and never exceeding 0.05 seconds. ILS with Seed requires an average of 15.64 seconds; in 55.73\% of instances it terminates in under one second, while the longest execution reaches 129.32 seconds (instance 131).

\end{document}